\documentclass{amsart}
\usepackage[utf8x]{inputenc}
\usepackage{amsmath, amsthm, amssymb}
\usepackage[usenames,dvipsnames,svgnames,table]{xcolor}
\usepackage[margin=1.26in]{geometry}
\usepackage{enumerate}
\usepackage{dsfont}
\usepackage{cleveref}
\usepackage{cite}

\newtheorem{theorem}{Theorem}[section]
\newtheorem{proposition}[theorem]{Proposition}
\newtheorem{lemma}[theorem]{Lemma}
\newtheorem{definition}[theorem]{Definition}
\newtheorem{remark}{Remark}

\newcommand{\BB}{\mathcal{B}}

\newcommand{\R}{\mathbb R}

\newcommand{\lin}{{\rm lin}}
\newcommand{\nlin}{{\rm nlin}}

\def\comma{ {\rm ,\qquad{}} }

\newcommand{\be}{\begin{equation}}
	\newcommand{\ee}{\end{equation}}

\numberwithin{equation}{section}

\title[Quantitative bounds of solutions]{Quantitative bounds for critically bounded solutions to the three-dimensional Navier-Stokes equations in Lorentz spaces}
\author{Wen Feng}
\address{(W. Feng) Department of Mathematics, Niagara University, Lewiston, NY 14109, USA}
\email{wfeng@niagara.edu}
\author{Jiao He}
\address{(J. He) Universit\'e Paris-Saclay, CNRS, Laboratoire de math\'ematiques d’Orsay, 91405, Orsay, France.}
\email{jiao.he@universite-paris-saclay.fr}
\author{Weinan Wang}
\address{(W. Wang) Department of Mathematics, University of Oklahoma, Norman, OK 73072, USA}
\email{ww@ou.edu}

\begin{document}
	\begin{abstract}
		In this paper, we prove a quantitative regularity theorem and blow-up criterion of classical solutions for the three-dimensional Navier-Stokes equations. 
		By adapting the strategy developed by Tao in \cite{tao}, we obtain an explicit blow-up rate in the setting of critical Lorentz spaces $L^{3, \mathfrak{q}_{0} }(\R^3)$ with $3 \leq \mathfrak{q}_0 < \infty $. Our results generalize the  quantitative regularity theory in critical Lebesgue spaces $L^3(\R^3)$ in \cite{tao} and quantify the qualitative result by Phuc in \cite{Phuc2015}.
	\end{abstract}
	
	\maketitle
\section{Introduction}
In this paper, we are interested in giving some quantitative bounds for solutions of the three-dimensional incompressible Navier-Stokes equations in critical Lorentz spaces. The Navier-Stokes equations read
		\begin{equation}\label{eq:NSE}
		\begin{cases}
		u_{t}-\Delta u +u\cdot \nabla u+\nabla p=0,
		\\
		\nabla \cdot u=0
		,
		\end{cases}
		\end{equation}
				where $u (t, \cdot) : \R^3 \to \R^3$ denotes the velocity vector field of the fluid and $p (t, \cdot) : \R^3 \to \R$ is the pressure. It is well-known from the seminal paper of Leray \cite{leray1934} that for any divergence-free vector field $u_0 \in L^2(\R^3)$ there exists at least one weak solution to the Cauchy problem \eqref{eq:NSE}. However, it is unknown whether they are smooth for all positive times and the uniqueness is also still open. The Navier-Stokes equations \eqref{eq:NSE} are endowed with a scaling symmetry:
		\begin{equation*}\label{invariance}
			u_{\lambda}(t,x):=\lambda u(\lambda^{2}t, \lambda x)
			\comma
			p_{\lambda}(t, x):=\lambda^{2} p( \lambda^{2}t, \lambda x)
			\quad \text{for}\quad \lambda >0,
		\end{equation*}
		which gives us some critical (scale-invariant) spaces, for example, $L^{3} (\mathbb{R}^3)$. A natural question that we are interested in is that, if we assume that blowing-up solutions do exist and they blow up at time $T^* >0$, how their critical norms behave and will they blow up as well at $T^*$? Besides the simplest critical Lebesgue spaces $L^3$, there are other critical spaces , such as critical Lorentz spaces $L^{3, q}  (\mathbb{R}^3)$ with $3<q<\infty $ and critical Besov spaces $\dot{B}^{-1+3/p}_{p, q} (\mathbb{R}^3)$ with $3 < p, q< \infty$, etc. In particular, we have a chain of embeddings
		 \begin{equation*}
			L^3(\mathbb{R}^3) \hookrightarrow L^{3,q}(\mathbb{R}^3)  \hookrightarrow \dot{B}^{-1+3/p}_{p, q}(\mathbb{R}^3)
			.
			\end{equation*}
			In the present paper, we investigate the quantitative estimate of the critical Lorentz norm $L^{3, \mathfrak{q_0}} (\mathbb{R}^3)$, $3 \leq \mathfrak{q}_0 < \infty$ at the potential singularity and the corresponding blow-up criterion. 

	Before introducing our main theorems, let us first present some previous results regarding to the regularity theory and blow-up criterion of the Navier-Stokes equations. The first result was given by Leray \cite{leray1934}, proving that if $T^*$ is the maximal existence time of the solution $u$, we then necessarily have for any $p>3$, there is a constant $C(p)$ such that
	$$\|u\|_{L^{p}(\R^3)}\geq \frac{C(p)}{(T^* -t)^{\frac{1}{2} (1 - \frac{3}{p})}}.$$
	Later, it was proved by Prodi-Serrin-Ladyzhenskaya (1959-1967) \cite{Prodi1959, Lady1967, Serrin1962} that if $u$ blows up at $T_*$ with $3<p\leq \infty$, then 
	$\|u\|_{L^{q}_{t}L_{x}^{p}([0,T_*)\times \R^3)}=\infty,$
	where $3/p+2/q=1$. 
	The endpoint case $p=3$ was left open for many years until the remarkable result of Escauriaza, Seregin, and Sverak \cite{ESS} in 2003. By analyzing the blow-up profile and combing the unique continuation with backward uniqueness of the heat equation, they were able to prove that if $u$ blows up at a finite time $T^*$, then 
		\begin{equation}\label{criterion_ESS3}
			\limsup_{t\rightarrow T^*}\|u(t)\|_{L_{x}^{3}( \R^3)}=
			\infty.
		\end{equation}
		Later, the blow-up criterion above has been generalized by several authors. On the one hand, in the qualitative sense, Seregin \cite{seregin2012} improved the blow-up criterion \eqref{criterion_ESS3} by replacing the limit superior with a limit. For the non-endpoint borderline Lorentz spaces, by applying the backward uniqueness theory as well as an $\epsilon$-regularity criterion, Phuc \cite{Phuc2015} proved that if a Leray-Hopf weak solutions $u$ blows up at a finite time $T^*$, then for $3<q<\infty$
		\begin{equation}\label{criterion_Lorentz}
			\limsup_{t\rightarrow T^*}\|u(t)\|_{L_{x}^{3,q}( \R^3)}=\infty.
		\end{equation}
	There are results in other critical spaces; see, for example, \cite{Albritton2018, GKP2016, DD2009}. 
	
	From the quantitative point of view, in a recent breakthrough work, by establishing quantitative Carleman inequalities, Tao \cite{tao} proved the following slightly supercritical blow-up norm criterion: 
		\begin{equation}\label{criterion_Tao}
			\limsup_{t\rightarrow T^{-}_{*}}\frac{\|u\|_{L^{3}(\R^3)}}{(\log\log\log\frac{1}{T_{*}-t})^c}
			=\infty,
		\end{equation} 
		which is a quantitative version of the $L^\infty L^3$ regularity criterion \eqref{criterion_ESS3}.  Barker and Prange \cite{BP2021} gave a quantitative estimate of the local concentration of $L^3$ norm by using an alternative proof.  
		The same authors also proved a mild supercritical regularity criteria, in which they showed that if a solution blows up, then certain slightly supercritical Orlicz norm must blow up \cite{BP2021mild}.
		Later, Palasek \cite{palasek2021improved} showed that Tao's blow-up rate can be improved to $(\log\log\frac{1}{T_{*}-t})^c$ assuming that the solution is axis-symmetric, and he recently obtained a $(\log \log\log\log\frac{1}{T_{*}-t})^c$ blow-up rate \cite{SP1} for the higher dimensional ($d\geq 4$) case. There are several other related results \cite{MR4646860, PO}.
	
In the setting of Lorentz spaces, there are few quantitative results. Davies and Koch \cite{davies2021local} recently gave a blow-up rate in sub-critical Lorentz spaces, in which they showed that if a solution $u$ blows up at finite time $T^*$, then 
\begin{equation}\label{criterion_Lorentz-sub}
			\|u\|_{L^{p,q}(\R^3)}\geq \frac{C(p, q)}{(T^* -t)^{\frac{1}{2} (1 - \frac{3}{p})}}\quad \text{for} \quad  3<p<\infty, 1 \leq q \leq \infty.
		\end{equation}  
To the best of the authors' knowledge, there is no such quantitative results in the critical Lorentz case, i.e. when $p = 3$ in \eqref{criterion_Lorentz-sub}.  
As the Lorentz space  $L^{p,q}$ has the same scaling properties as the Lebesgue space $L^{p}$, 
Tao's results \cite{tao} for $L^3$ open the door to treat quantitatively the critical Lorentz spaces $L^{3, q}, q \geq 3$, which are bigger than the usual Lebesgue spaces. 
The main goal of the present paper is to obtain new quantitative regularity theorem and blow-up criteria of solutions for the Navier-Stokes equations in the framework of \textit{critical Lorentz spaces} $L^{3, \mathfrak{q}_{0} }(\R^3)$ for $3\leq \mathfrak{q}_{0} <\infty$.  
Our main results are stated as follow. 
\begin{theorem}\label{t.w08292}
Let $(u,p)$ be a classical solution to the incompressible Navier-Stokes system \eqref{eq:NSE}, which blows up at time $T_{*}<\infty$. Then, with a constant $c>0$ and $3\leq \mathfrak{q}_{0} <\infty$
\begin{equation}\label{criteria_us}
				\limsup_{t\rightarrow T^{-}_{*}}\frac{\|u\|_{L_x^{3, \mathfrak{q}_{0} }(\R^3)}}{(\log\log\log\frac{1}{T_{*}-t})^c}
				=
				\infty
				.
\end{equation}
\end{theorem}
\begin{theorem}\label{t.w08291}
Let $(u,p)$ be a classical solution to the system \eqref{eq:NSE} and $3\leq \mathfrak{q}_{0} <\infty$. Assume that
			\begin{equation*}
			\|u\|_{L_{t}^{\infty}L_{x}^{3, \mathfrak{q}_{0}}([0,T]\times\mathbb R^3)}
			\lesssim M 
			\end{equation*}
			for some constant $M\geq 2$. Then, for $0<t\leq T$ and $j=0,1$, the following hold
			\begin{equation*}
			|\nabla^{j}u|
			\leq
			\exp{\exp{\exp{(M^{O(1)})}}}
			t^{-\frac{j+1}{2}} \comma
			|\nabla^{j}\omega|
			\leq
			\exp{\exp{\exp{(M^{O(1)})}}}
			t^{-\frac{j+1}{2}},
    		\end{equation*}
			where the vorticity $\omega=\nabla \times u$.
	\end{theorem}
\begin{remark}
Notice that when $\mathfrak{q}_{0} = 3$, our theorem reduces to the case of Tao in \cite{tao}. 
\end{remark}

\begin{remark}
It is not clear whether our results hold in the endpoint case $L_{t}^{\infty}L_{x}^{3,\infty}$. For other results in such spaces, we refer to \cite{BP2020} and \cite{BP2021} for quantitative results in local sense and \cite{Luo-Tsai2013} for qualitative result. 
\end{remark}	
Comparing the blow-up criteria \eqref{criteria_us} in Theorem \ref{t.w08292} with the previous results, we see that our theorem implies that the necessary condition of blow up \eqref{criterion_Tao} can be improved by replacing the $L^3$ norm with a smaller $L^{3, \mathfrak{q}_{0} }(\R^3)$ quasi-norm with $3\leq \mathfrak{q}_{0}<\infty$, answering a question of Tao, see Remark 1.6 in \cite{tao}.
		In particular, we recover the result \eqref{criterion_Tao} when $\mathfrak{q}_{0}=3$. Moreover, our result can be seen as a quantitative version of the blow-up criteria \eqref{criterion_Lorentz} proved by Phuc in \cite{Phuc2015}. 
		Let us remark that although the blow-up rate \eqref{criterion_Lorentz-sub} is better than the one we established in Theorem \ref{t.w08292}, it is not clear whether their results still hold in the critical case.
	

	
The rest of this paper is organized as follows. In Section \ref{notation}, we introduce some notation and preliminaries. Furthermore, we prove H\"older's and Young's inequalities in the Lorentz setting. In Section \ref{estimate}, we state our main quantitative estimates (Propositions \ref{prop : bounded}-\ref{p.w08294}). We demonstrate how to gather these statements together to prove our main results (Theorem \ref{t.w08292} and Theorem \ref{t.w08291}) of this paper in Section \ref{proofs}. More precisely, our strategy is as follows : first, by establishing some pointwise derivative estimates, bounded total speed and Epoch of regularity (see Proposition \ref{prop : bounded}-\ref{p.w08292}), we are able to prove that, assuming $
\|u\|_{L_{t}^{\infty}L_{x}^{3, \mathfrak{q}_{0}}([0,T]\times\mathbb R^3)}
\leq M$ is such that $N_0^{-1} |P_{N_0} u(x_0, t_0)| > M^{-O(1)}$, we create a chain of ``bubbles of concentration" (see Propositions \ref{p.w08293}-\ref{p.w08294}). Then, using a similar procedure as in Tao (p.36-p.41 in \cite{tao}), we show the lower bound of $N_0$ (see Theorem \ref{t.w08293}). Once we obtain such lower bound, we show the first main result (see Theorem \ref{t.w08291}) via a contrapositive argument, and then the second main result (see Theorem \ref{t.w08292}) is proven by contradiction. Lastly, we include the Carleman estimates together with auxiliary estimates in the Appendix.

	\section{Notation and preliminaries}\label{notation}
		\subsection{Notation}
		Throughout the paper we use the following notation. 
		For $1\leq p<\infty$,  $W^{k,p}$ space is the regular Sobolev space. 
		We have Plancherel's equality, $\|f\|_{L^2}=\|\hat{f}\|_{L^2}$.
		For any $N\geq 0$, we define the Littlewood-Payley projection $P_{\leq N}$
		\begin{equation*}
		\widehat{P_{\leq N}f}(\xi)=\varphi(\xi/N)\hat{f}(\xi),
		\end{equation*}
		where $\varphi$ is a smooth bump function on the ball $B(0,1)$ with $\varphi=1$ on $B(0,1/2)$. Then, we define
$$
		P_{ N}:=P_{\leq N}-P_{\leq  N/2} \comma
		P_{ >N}:=1-P_{\leq N} \comma
		\widetilde{P}_{ N}:=P_{\leq  N/2}-P_{\leq  N/4}. $$
		Therefore, $P_{\leq N}f=\sum_{k=0}^{\infty}P_{2^{-k} N}f$ and $P_{> N}f=\sum_{k=1}^{\infty}P_{2^{k} N}f$.
We remark here that these operators commute with other Fourier multipliers such as $\Delta$, $e^{t \Delta}$ and the Leray projector $\mathbb P$ defined by
		\begin{equation*}
			\mathbb P
			=
			I + \nabla (-\Delta)^{-1} \nabla \cdot.
		\end{equation*} 
		Next, we define the Lorentz space. 
		\begin{definition}\label{lorentz}
			For a measurable function $f: \Omega \rightarrow \mathbb R$, we define:
		\begin{equation*}
			d_{f,\Omega}(\alpha):=
			|\{x\in \Omega: |f(x)|>\alpha\}|
			.
		\end{equation*}
			Then, the Lorentz spaces $L^{p,q}(\Omega)$ with $1\leq p<\infty$, $1\leq q \leq \infty$ is the set of all functions $f$ on $\Omega$ such that the quasinorm $\|f\|_{L^{p,q}(\Omega)}$ is finite and
			\begin{equation*}
			\|f\|_{L^{p,q}(\Omega)}
			:=
			\left(p\int_{0}^{\infty}\alpha^{q} d_{f,\Omega}(\alpha)^{\frac{q}{p}}\frac{d\alpha}{\alpha}\right)^{1/q}
			\comma
			\|f\|_{L^{p,\infty}(\Omega)}
			:=
			\sup_{\alpha>0} \alpha d_{f,\Omega}(\alpha)^{1/p}
			.
			\end{equation*}
		\end{definition}
		When we say $A \lesssim B$, it means there is a constant $C>0$ such that $A\leq CB$. The space $L^{p, \infty}$ is known as the weak $L^p$ space and notice that when $q=p$, we have $\|f\|_{L^{p,p}(\Omega)}=\|f\|_{L^{p}(\Omega)}$ and $L^{p,q_1}(\Omega) \subset L^{p, q_2} (\Omega)$ whenever $1 \leq q_1 \leq q_2 \leq \infty$.
	\subsection{Preliminaries}
	The next lemma is H\"older's inequality in Lorentz spaces (Theorem 4.5 in \cite{hunt1966p}).
		\begin{lemma}[H\"older's inequality,  \cite{hunt1966p}] \label{Lemma-Holder}
			Suppose $f \in L^{r_1, s_1} (\R^3)$, $g \in L^{r_2, s_2} (\R^3)$ with $0 < r_1, r_2, r < \infty$, $0 < s_1, s_2, s \leq \infty$,
			\begin{equation*}\label{ineq : Holder}
				1/r=1/r_{1}+1/r_{2}
				\quad
				\text{and}
				\quad
				1/s = 1/s_{1}+1/s_{2}
			\end{equation*}
			Then $fg \in L^{r,s}(\R^3)$ and 
			\begin{equation*}
			\|fg\|_{L^{r,s}(\R^3)}
			\leq
			C(r_1, r_2, s_1, s_2)
			\|f\|_{L^{r_{1},s_{1}}(\R^3)}
			\|g\|_{L^{r_{2},s_{2}}(\R^3)}.
			\end{equation*}
		\end{lemma}
		The next lemma is Young's convolution inequality in Lorentz spaces, also known as ``O'Neil's convolution inequality" (Theorem 2.6 of O'Neil's paper \cite{o1963convolution}).
		\begin{lemma}[Young's inequality, \cite{o1963convolution}]\label{Lemma_young}
			Suppose $f \in L^{r_1, s_1} (\R^3)$, $g \in L^{r_2, s_2} (\R^3)$ with $1 < r_1, r_2, r < \infty$, $0 < s_1, s_2, s \leq \infty$,
			\begin{equation*}
				1/r+1=1/r_{1}+1/r_{2}
				\quad
				\text{and}
				\quad
				1/s\leq 1/s_{1}+1/s_{2}
			\end{equation*}
			Then $f *g \in L^{r,s}(\R^3)$ and 
			\begin{equation*}
				\|f*g\|_{L^{r,s}(\R^3)}
				\leq
				3r \|f\|_{L^{r_{1},s_{1}}(\R^3)}
				\|g\|_{L^{r_{2},s_{2}}(\R^3)}.
			\end{equation*}
		\end{lemma}
		\begin{lemma}[Sobolev's inequality, \cite{tartar1998imbedding}]\label{l:jh} Suppose $1 \leq p \leq 3 $, 
			then 
			\begin{equation*}
				\|f\|_{L^{\frac{3p}{3-p},p}(\R^3)}
				\leq
				C(p) \|\nabla f\|_{L^{p}(\R^3)}.
			\end{equation*}
		\end{lemma}
		The following lemma is relied on lemma 2.4 in \cite{wang2021gagliardo}, in which the authors gave a Bernstein inequality for weak $L^p$ spaces. We now state a generalized Bernstein inequality for general Lorentz spaces $L^{p,q}$ with $p >1, q \geq 1$.
		\begin{lemma}[Bernstein inequality] Let a ball $\BB = \{ \xi \in \R^3 : |\xi| \leq R\}$ with $ 0< R < \infty$. Then there exists a constant $C$ such that for any non-negative integer $j$, any couple $(p_1, p_2)$ with $1 < p_1 < p_2 < \infty$, for any $N \in (0, \infty)$, and any function $f$ of $L^{p_2, q_2}$ with $ 1 \leq q_2\leq \infty$, whose Fourier transform is in the support of the ball $\BB(0,N)$, we have 
			\begin{equation}
				\label{e:bernstein}
				\|\nabla^{j}f\|_{L^{p_1, q_1}(\mathbb R^3)}
				:= \sup_{|\alpha| = j} \| \partial^\alpha f\|_{L^{p_1, q_1}(\mathbb R^3)}
				\lesssim_j
				N^{j+3(\frac{1}{p_2}-\frac{1}{p_1})}
				\|f\|_{L^{p_2, q_2}(\mathbb R^3)}, \quad j \geq 0
				.
			\end{equation}
			In particular, 
			\begin{equation*}
				\label{e:bernstein-Leb}
				\|\nabla^{j}f\|_{L^{p_1}(\mathbb R^3)}
				\lesssim_j
				N^{j+3(\frac{1}{p_2}-\frac{1}{p_1})}
				\|f\|_{L^{p_2}(\mathbb R^3)}.
			\end{equation*}
		\end{lemma}
		\begin{proof}
			From lemma 2.4 in \cite{wang2021gagliardo}, we have 
			$$
			\|\nabla^{j}f\|_{L^{p_1, 1}(\mathbb R^3)} 
			\lesssim  N^{j+3(\frac{1}{p_2}-\frac{1}{p_1})}
			\|f\|_{L^{p_2, \infty}(\mathbb R^3)}, 
			$$
			Hence, by inclusion property of Lorentz spaces $L^{p, 1} \subset L^{p, q} \subset L^{p, \infty}$ for $1<q<\infty$, we obtain
			\begin{align*}
				& \|\nabla^{j}f\|_{L^{p_1, q_1}(\mathbb R^3)}
				\leq \|\nabla^{j}f\|_{L^{p_1, 1}(\mathbb R^3)} 
				\lesssim  N^{j+3(\frac{1}{p_2}-\frac{1}{p_1})}
				\|f\|_{L^{p_2, \infty}(\mathbb R^3)} 
				\leq 
				N^{j+3(\frac{1}{p_2}-\frac{1}{p_1})}
				\|f\|_{L^{p_2, q_2}(\mathbb R^3)}.
			\end{align*}
		\end{proof}
		We state a generalized multiplier theorem for Lorentz spaces as follows. 
		\begin{lemma}[Multiplier theorem]
			Let $T_m$ be a Fourier multiplier $\widehat{T_m f} (\xi) := m(\xi) f(\xi)$ where $m (\xi)$ is a complex-valued smooth function supported on $\BB (0, N)$ satisfying 
			$$
			|\nabla^j m(\xi)| \leq A N^{-j}
			$$
			for some positive $A$ and $j > 0$. Then we have
			\begin{equation}\label{e:multiplier1}
				\|T_m f\|_{L^{p_1, q_1} (\R^3)} \lesssim 
				A N^{3(\frac{1}{p_2} - \frac{1}{p_1})} \|f\|_{L^{p_2, q_2}(\R^3)},
			\end{equation}
			where $\frac{1}{p_2}+\frac{1}{q_1} \leq \frac{1}{p_1}+\frac{1}{q_2} + 1$.
			Moreover, let $D \subset \R^3$ be a subset of $\R^3$ and $D_{R/N} := \{x \in \R^3 : \text{dist} (x, D) < R/N \} $ be the $R/N$-neighbourhood of $D$, then we have  
			\begin{equation}\label{e:multiplier2}
				\|T_m f\|_{L^{p_1, q_1} (D)} \lesssim 
				A N^{3(\frac{1}{p_2} - \frac{1}{p_1})} \|f\|_{L^{p_2, q_2}(D_{R/N})} + R^{-50} A |D|^{\frac{1}{p_1} - \frac{1}{p_4}} N^{3(\frac{1}{p_3} - \frac{1}{p_4})} \|f\|_{L^{p_3, q_3}(\R^3)}
				,
			\end{equation}
			where $|D|$ denotes the volume of set $D$ and $1 \leq p_2 \leq p_1 \leq \infty$, $1 \leq p_3 \leq p_4 \leq \infty$ such that $p_4 \geq p_1$. 
		\end{lemma}
		\begin{proof}
			Let us first write $T_m f$ as a convolution $T_m f = f * K$ with the kernel 
			$$
			K(x) = \int_{\R^3} m(\xi) e^{2 \pi i \xi \cdot x} d \xi.
			$$
			By Young's inequality in Lorentz spaces (see Lemma \ref{Lemma_young}), we have 
			\begin{align*}
				\|T_m f\|_{L^{p_1, q_1} (\R^3)} 
				= \|f * K\|_{L^{p_1, q_1} (\R^3)}
				&\leq 3 p_1 \|f \|_{L^{p_2, q_2} (\R^3)}
				\|K\|_{L^{p} (\R^3)},
			\end{align*}
			where $\frac{1}{p_1}+1=\frac{1}{p_{2}}+\frac{1}{p}$ and $\frac{1}{q_1} \leq \frac{1}{q_{2}}+\frac{1}{p}$.
			Applying Plancherel's theorem, we get $$\|K (x)\|_{L^{2} (\R^3)} = \|m(\xi)\|_{L^{2} (\R^3)} \leq A N^{3/2}.$$ On the other hand, by the bound of $m(\xi)$, we have $\|K (x)\|_{L^{\infty} (\R^3)} \leq A N^3$. Combining the two estimates and using interpolation inequality $\|K\|_{L^p} \leq \|K\|_{L^2}^{\frac{2}{p}} \|K\|_{L^\infty}^{1-\frac{2}{p}}$, we can conclude that
			$$\|T_m f\|_{L^{p_1, q_1} (\R^3)} \lesssim 
			A N^{3(1 - \frac{1}{p})} \|f\|_{L^{p_2, q_2}(\R^3)}
			\lesssim 
			A N^{3(\frac{1}{p_2} - \frac{1}{p_1})} \|f\|_{L^{p_2, q_2}(\R^3)}.$$
			Now, we prove the local version estimate \eqref{e:multiplier2}. We have
			\begin{align*}
				T_m f
				= &\int_{\R^3} K(x-y) f(y) dy \; \\
				= & \int_{D_{R/N}} K(x-y) f(y) dy \;  + \int_{D^c_{R/N}} K(x-y) f(y) dy. \; 
			\end{align*}
			For the first term in the right-hand side, applying the global estimate \eqref{e:multiplier1}, we get that 
			\begin{align*}
				\left\|\int_{D_{R/N}} K(\cdot-y) f(y) dy \; \right\|_{L^{p_1, q_1} (D)}
				= & \left\|\int_{D_{R/N}} K(\cdot-y) f(y) dy \; \mathbf{1}_D (\cdot)\right\|_{L^{p_1, q_1} (\R^3)} \\
				\lesssim &  A N^{3(\frac{1}{p_2} - \frac{1}{p_1})} \|f\|_{L^{p_2, q_2}(D_{R/N})}.
			\end{align*}
			For the second term, set $\widetilde{K} (z) = K(z) \mathbf{1}_{|z| \geq R/N}$, by H\"older's inequality, change of variable, and Young's inequality, we have
			\begin{align*}
				\left\|\int_{D^c_{R/N}} K(\cdot-y) f(y) dy \; \right\|_{L^{p_1, q_1} (D)}
				\leq & 
				|D|^{1/p_{1} - 1/p_{4}} \left\|\int_{D^c_{R/N}} K(\cdot-y) f(y) dy \; \mathbf{1}_D (\cdot)\right\|_{L^{p_4, q_1} (\R^3)} \\
				= & |D|^{1/p_{1} - 1/p_{4}}  \| \widetilde{K} * f \|_{L^{p_4, q_1} (\R^3)} \\
				\leq & |D|^{1/p_{1} - 1/p_{4}}  \|\widetilde{K}\|_{L^{p, q} (\R^3)} \|f\|_{L^{p_3, q_3} (\R^3)},
			\end{align*}
			where $1 + 1/p_{4} = 1/p + 1/p_{3}$. We compute $\|\widetilde{K}\|_{L^{p, q} (\R^3)}$ to get
			\begin{align*}
				\|\widetilde{K}\|_{L^{p, q} (\R^3)}
				= \|K (\cdot) \mathbf{1}_{|\cdot| \geq R/N}\|_{L^{p, q} (\R^3)} 
				\lesssim R^{-50} A N^{3(\frac{1}{p_3} - \frac{1}{p_4})}.
			\end{align*}
			This concludes the proof of \eqref{e:multiplier2}.
		\end{proof}
		By writing $f=\sum P_{\leq 2N}f$, applying interpolation theorem and Young's inequality, for any $t >0$ and any Schwartz function $f$, we can prove that
		\be
		\label{e:bersteinpn}
		\|P_{N}e^{t\Delta}\nabla^{j}f\|_{L^{p_1, q_1}(\R^3)}
		\lesssim_j
		e^{-\frac{N^{2}t}{20}}N^{j+3(\frac{1}{p_2}-\frac{1}{p_1})}
		\|f\|_{L^{p_2, q_2}(\R^3)}
		\ee		
		with $p_1\leq p_2$.
		The following heat kernel bounds in Lorentz spaces can be derived by summing the inequality above over $N$, 
		\be 
		\label{e:bernsteintime}
		\|e^{t\Delta}\nabla^{j}f\|_{L^{p_1, q_1}(\R^3)}
		\lesssim_j
		t^{-\frac{j}{2} -\frac{3}{2}(\frac{1}{p_2}-\frac{1}{p_1})}
		\|f\|_{L^{p_2, q_2}(\R^3)}
		\ee	
		with $p_1\leq p_2$.
	
	\section{Basic estimates}\label{estimate}	
	In this section, we use the notation 
	\begin{equation*}
		M_j:=M^{C_0^j}
	\end{equation*}
	for all integers $j \geq 0$, thus $M_0=M$ and $M_{j+1}=M_j^{C_0}$. As the following assumption will be used several times through our our paper, we call it \eqref{e:ub-bounds}
		\begin{equation}
			\tag{HP}
			\label{e:ub-bounds}
			\|u\|_{L_t^{\infty} L^{3, \mathfrak{q}_{0}}_x([t_0-T, t_0]\times \mathbb{R}^3)}  \leq M. 
		\end{equation}
		for some $M\geq C_0$. 
	\begin{proposition}\label{prop : bounded}
		Let $u :[t_0-T, t_0] \times \mathbb{R}^3 \to \mathbb{R}^3$, $p: [t_0-T, t_0]\times \mathbb{R}^3 \to \mathbb{R}$ be a classical solution to Navier-Stokes that obeys (HP), then
		\begin{itemize}
			\item [(i)] (Pointwise derivative estimates) For any $(t, x)\in [t_0-T/2, t_0] \times \mathbb{R}^3$ and $N>0$, we have
			\begin{align}
				\begin{split}
					\label{e:pnub-bounds}
					&P_Nu(t,x)=O(MN); \quad \quad \nabla P_N u(t, x) =O(MN^2); \quad \quad \partial_t P_N u(t,x)=O(M^2N^3);  
				\end{split}
			\end{align}
			similarly, the vorticity $\omega:= \nabla \times u$ obeys the bounds 
			\begin{align}
				\begin{split}
					\label{e:pnomjbounds}
					&P_N\omega(t,x)=O(MN^2); \quad \quad \nabla P_N \omega(t, x) =O(MN^3); \quad \quad \partial_t P_N \omega(t,x)=O(M^2N^4); 
				\end{split}
			\end{align} 		
			\item [(ii)] (Bounded total speed) For any interval $I$ in $[t_0-T/2, t_0]$, we have 
			\begin{equation*}
				\label{e:totalspeed}
				\|u\|_{L_t^1L_x^{\infty}(I\times \mathbb{R}^3)} \lesssim M^4|I|^{1/2}. 
			\end{equation*}					
		\end{itemize}		
	\end{proposition}	
\begin{proof}	
	We start with the proof of (i). By \eqref{e:ub-bounds} and \eqref{e:bernstein}, we can obtain the first two claims of \eqref{e:pnub-bounds} and \eqref{e:pnomjbounds}. 
	After applying the Leray projector to equation \eqref{eq:NSE}, 
	we get
		\begin{align*}
			\begin{cases}
				\partial_t u
				-
				\Delta u
				+
				\mathbb P \nabla \cdot (u \otimes u)
				=
				0, 
				\\
				\nabla \cdot u=0.
			\end{cases}
		\end{align*}
		Then, by Duhamel's formula, we obtain
		\begin{align}\label{e.w07011}
			u(x,t)=
			e^{t\Delta}u_{0}-\int_{0}^{t} e^{(t-s)\Delta} \left(\mathbb P \nabla \cdot (u \otimes u)\right)\,ds. 
		\end{align}
		Here we see
		\begin{align*}
			(\nabla \cdot (u\otimes u))_{j}=\partial_{i}(u_{i}u_{j}), 
		\end{align*}
		Apply $P_{N}$ to equations \eqref{e.w07011} and we get
		\begin{align*}
			\|P_{N}\Delta u\|_{L_{x}^{\infty}}
			\lesssim 
			N^{3}M.
		\end{align*}
		Furthermore, by the multiplier theorem and H\"older's inequality (in both Lebesgue and Lorentz spaces), we obtain
		\begin{equation*}\label{e.w07013}
			\| u \otimes u\|_{L_{x}^{3/2,\mathfrak{q}_{0}/2}}
			\lesssim
			\| u \|_{L_{x}^{3,\mathfrak{q}_{0}}}^{2}
			\lesssim
			M^{2}.
		\end{equation*}
		And thus by \eqref{e:bernstein} ($p_1=q_1=\infty$, $j=1$, $p_2=\frac{3}{2}$), 
		\begin{equation*}\label{e.w07014}
			\|P_{N}\mathbb P \nabla \cdot( u \otimes u)\|_{L_{x}^{\infty}}
			\lesssim 
			N^3
			\| u \otimes u\|_{L_{x}^{3/2,\mathfrak{q}_{0}/2}}
			\lesssim
			N^{3}M^2.
		\end{equation*}
		By the triangle inequality, we obtain the third and six claims of \eqref{e:pnub-bounds} and \eqref{e:pnomjbounds}. 	
		Then we prove (ii). Since these estimates are invariant with respect to time translation and rescaling (adjusting $T, t_0, I, u, b$ accordingly), without loss of generality, we assume that $I=[0, 1] \subset [t_0-T/2, t_0]$, which implies that $[-1, 1]\subset [t_0-T/2, t_0]$. 
		Next, we decompose $(u,b)$ into linear and nonlinear parts. The reason we do this is that by removing linear components from $(u,b)$, we will have better control in $L^{2}_{x}$ based spaces. Thus, we see
		\begin{align*}
			u=u^{\lin}+u^{\nlin}, 
		\end{align*}
		where $(u^{\lin},b^{\lin})$ are linear solutions on $[-1,1]\times \R^{3}$
		\begin{align}
			\label{eq:linpart}
			u^{\lin}(x,t)=e^{(t+1)\Delta}u(x,-1).
		\end{align}
		Then we have 
		\begin{align*}
			\nabla \cdot u^{\lin}=0.
		\end{align*}
	By assumption \eqref{e:ub-bounds}, we see
	\be
	\label{eq:linbounds}
	\|u^{\lin}\|_{L_{t}^{\infty}L_{x}^{3, \mathfrak{q}_{0}}([-1,1]\times\mathbb R^3)}
	+
	\|u^{\nlin}\|_{L_{t}^{\infty}L_{x}^{3, \mathfrak{q}_{0}}([-1,1]\times\mathbb R^3)}
	\lesssim
	M
	.
	\ee
		Thus, on $[-1,1]\times \R^{3}$, we have
		\begin{align}\label{eq : u_nlin}
			\partial_t u^{\nlin}
			-
			\Delta u^{\nlin}
			+
			u\cdot \nabla u
			+
			\nabla p
			=
			0, 
			\\
			\nabla \cdot u^{\nlin}=0.
		\end{align}
		We obtain by using Duhamel's formula
		\begin{align}\label{e.w06301}
			u^{\nlin}(x,t)=
			-\int_{-1}^{t} e^{(t-s)\Delta} \left(\mathbb P \nabla \cdot (u \otimes u)
			\right)\,ds.
		\end{align}
	From \eqref{e:ub-bounds}, $u\otimes u$ has an $L_x^{3/2, \mathfrak{q}_{0}/2}$ norm of $O(M^2)$ and  by \eqref{e:bernsteintime}, 
	\begin{equation*}
		\|e^{(t-s)\Delta} \mathbb P \nabla \cdot (u \otimes u)\|_{L^{2,2}} \leq (t-s)^{-\frac{1}{2}-\frac{3}{2}(\frac{2}{3}-\frac{1}{2})}\|u\otimes u\|_{L^{3/2, \mathfrak{q}_{0}/2}} \leq (t-s)^{-\frac{3}{4}}\|u\|_{L^{3,\mathfrak{q}_{0}}}^2. 
	\end{equation*}
	We conclude a bound for the nonlinear part:
		\begin{align}\label{e.w07016}
			\|u^{\nlin}\|_{L_{t}^{\infty}L_{x}^{2}}
			\lesssim
			M^{2}
			.
		\end{align}	
	By the hypothesis \eqref{e:ub-bounds}, equations \eqref{e:bernsteintime}, and \eqref{eq:linpart}, we obtain 
	\begin{align}\label{e.w07017}
		\|\nabla^{j}u^{\lin}\|_{L_{t}^{\infty}L_{x}^{p,q_1} ([-1/2, 1] \times \R^3)}
		\lesssim
		M,
	\end{align}
	where $j\geq 0$ and $3\leq p \leq \infty$, $\frac{1}{q_1}\leq \frac{1}{p}+\frac{1}{\mathfrak{q}_{0}}+\frac{2}{3}$. 
	Next, we consider the energy method on the $u^{\nlin}$ equation. We do an $L^{2}$ estimate.
		\begin{align*}
			\frac{1}{2}\frac{d}{dt}\|u^{\nlin}\|^{2}_{L^{2}(\R^3)}
			+
			\|\nabla u^{\nlin}\|^{2}_{L^{2}(\R^3)}
			=
			\int (\nabla  u^{\nlin}) \cdot (u\otimes u)\,dx
			.
		\end{align*}
		Due to the nature of $u^{\nlin}$ being divergence-free, we see
		\begin{align*}
			\int (\nabla  u^{\nlin}) \cdot (u^{\nlin}\otimes u^{\nlin})\,dx
			=0
			.
		\end{align*}
		Thus,
		\begin{equation*}
			\begin{split}
				\frac{1}{2}\frac{d}{dt}\|u^{\nlin}\|^{2}_{L^{2}(\R^3)}
				+
				\|\nabla u^{\nlin}\|^{2}_{L^{2}(\R^3)}
				&=
				\int_{\R^3} (\nabla  u^{\nlin}) \cdot (u\otimes u-u^{\nlin}\otimes u^{\nlin})\,dx
				\\&\leq
				\frac{1}{2}\|\nabla u^{\nlin}\|^{2}_{L^{2}(\R^3)}
				+
				2\|u\otimes u-u^{\nlin}\otimes u^{\nlin}\|^{2}_{L^{2}(\R^3)}
				.
			\end{split}
		\end{equation*}
		Then, integrating on the time interval $[-1/2,1]$ yields
		\begin{equation*}
			\begin{split}
				\int_{-1/2}^{1} \int_{\R^3} |\nabla  u^{\nlin}|^{2}\,dx \,dt
				\leq
				M^{4}
				+
				4\int_{-1/2}^{1}\|u\otimes u-u^{\nlin}\otimes u^{\nlin}\|^{2}_{L^{2}(\R^3)}\,dt
				.
			\end{split}
		\end{equation*}
		Notice that
		\begin{equation*}\label{e.w06294}
			u\otimes u-u^{\nlin}\otimes u^{\nlin}=
			u^{\lin}\otimes u
			+
			u^{\nlin} \otimes u^{\lin}
			.
		\end{equation*}
	Then,  by hypothesis \eqref{e:ub-bounds}, \eqref{eq:linbounds}, \eqref{e.w07017} and H\"older's inequality, 
	\begin{align*}
		\begin{split}
			&\int_{-1/2}^{1}
			\|u\otimes u-u^{\nlin}\otimes u^{\nlin}\|^{2}_{L^{2}(\R^3)}\,dt
			=
			\int_{-1/2}^{1}\|u^{\lin}\otimes u
			+
			u^{\nlin} \otimes u^{\lin}\|^{2}_{L^{2}(\R^3)}\,dt
			\\
			&\leq
			\int_{-1/2}^{1}\|u^{\lin}\otimes u\|^{2}_{L^{2}(\R^3)}\,dt
			+
			\int_{-1/2}^{1}\|
			u^{\nlin} \otimes u^{\lin}\|^{2}_{L^{2}(\R^3)}\,dt
			\\
			&=
			\|u^{\lin}\otimes u\|^{2}_{L^{2}_{t}L_{x}^{2}([-1/2,1]\times \R^3)}
			+
			\|u^{\nlin}\otimes u^{\lin}\|^{2}_{L^{2}_{t}L_{x}^{2}([-1/2,1]\times \R^3)}
			\\&\leq
			\|u^{\lin}\|^{2}_{L^{\infty}_{t}L_{x}^{6,q_1}([-1/2,1]\times \R^3)}
			\| u\|^{2}_{L^{2}_{t}L_{x}^{3, \mathfrak{q}_0 }([-1/2,1]\times \R^3)}
			\\
			&\quad +
			\|u^{\nlin}\|^{2}_{L^{2}_{t}L_{x}^{3, \mathfrak{q}_0 }([-1/2,1]\times \R^3)}
			\| u^{\lin}\|^{2}_{L^{\infty}_{t}L_{x}^{6,q_1}([-1/2,1]\times \R^3)}
			\lesssim M^4,
		\end{split}
	\end{align*}
	where we assumed $1/2=1/\mathfrak{q}_{0}+1/q_{1}$. and $1/q_{1}\leq 5/6+1/\mathfrak{q}_{0}$, which is equivalent to $q_1\geq 3/2$. Thus, 
		\begin{equation*}\label{e.w06295}
			\int_{-1/2}^{1} \int_{\R^3}  |u\otimes u-u^{\nlin}\otimes u^{\nlin}|^{2}\,dx \,dt
			\lesssim
			M^{4}
			.
		\end{equation*}
		Combining the above estimates yields
		\begin{align}\label{e.w07019}
			\int_{-1/2}^{1} \int_{\R^3} |\nabla  u^{\nlin}|^{2} \,dx \,dt
			\lesssim
			M^{4}.
		\end{align}
		By Sobolev embedding in Lemma \ref{l:jh}, we see for $q\geq 2$
		\begin{align}\label{e.w070111}
			\| u^{\nlin}\|_{L_{t}^{2}L_{x}^{6,q}([-1/2,1]\times \R^3)}
			\lesssim
			\| u^{\nlin}\|_{L_{t}^{2}L_{x}^{6,2}([-1/2,1]\times \R^3)}
			\lesssim
			\| \nabla u^{\nlin}\|_{L_{t}^{2}L_{x}^{2}([-1/2,1]\times \R^3)}
			\lesssim
			M^{2}
			.
		\end{align}
		By Plancherel's theorem, we get
		\begin{align}\label{e.w070110}
			\sum_{N}N^{2}	\|P_{N}u^{\nlin}\|^{2}_{L_{t}^{2}L_{x}^{2}([-1/2,1]\times \R^3)}	
			\lesssim
			M^{4}
			.
		\end{align}
		Next, we prove the total speed property. Recall equation \eqref{e.w06301}, apply the Littlewood-Paley projector $P_{N}$ and get
		\begin{align}\label{e.w06302}
			P_{N}u^{\nlin}(x,t)=
			e^{(t+\frac{1}{2})\Delta}P_{N}u^{\nlin}(-\tfrac{1}{2})
			-\int_{-1/2}^{t} P_{N}e^{(t-s)\Delta} \left(\mathbb P \nabla \cdot \tilde{P}_{N}(u \otimes u)\right)\,ds.
		\end{align}
		We will have
		\begin{equation}\label{e.w06303}
			\begin{split}
				\|P_{N}u^{\nlin}\|_{L_{t}^{1}L_{x}^{\infty}([0,1]\times \R^3)}
				&\lesssim
				MN\exp{(-N^{2}/20)}
				+
				N^{-1}\|\tilde{P}_{N}(u\otimes u)\|_{L_{t}^{1}L_{x}^{\infty}([-1/2,1]\times \R^3)}
				.
			\end{split}
		\end{equation}
	Indeed, by equation \eqref{e:bersteinpn} ($p_1=q_1=\infty$, $p_2=3$), we see
	\begin{align*}
		\left\|e^{(t+\tfrac{1}{2})\Delta}P_{N}u^{\nlin}(-\tfrac{1}{2})\right\|_{L_{t}^{1}L_{x}^{\infty}([0,1]\times \R^3)} \leq 	\left\|e^{-N^2/20(t+\frac{1}{2})}N \|u^{\nlin}(-\tfrac{1}{2})\|_{L_x^{3, { \mathfrak{q}_0 }}}\right\|_{L_t^1} 
		\leq 	MN\exp{(-N^{2}/20)}.
	\end{align*}
	Further, by \eqref{e:bersteinpn}, 
	\begin{align*}
		&\left\|\int_{-1/2}^{t} P_{N}e^{(t-s)\Delta} \left(\mathbb P \nabla \cdot \tilde{P}_{N}(u \otimes u)\right)\,ds\right\|_{L_{t}^{1}L_{x}^{\infty}([0,1]\times \R^3)} \\
		&\lesssim 
		\int_{0}^{1}\left(\int_{-1/2}^{t}Ne^{-N^2(t-s)/20}~ds\right) \|\tilde{P}_{N}(u \otimes u)\|_{L_x^{\infty}}\,dt\\
		&\lesssim 
		N^{-1}\|\tilde{P}_{N}(u \otimes u)\|_{L_{t}^{1}L_{x}^{\infty}([-1/2,1]\times \R^3)}.
	\end{align*}
		Next, we split
		\begin{equation*}\label{e.w07015}
			u\otimes u
			=
			u^{\lin}\otimes u^{\lin}+u^{\nlin}\otimes u^{\lin}+u^{\lin}\otimes u^{\nlin}+u^{\nlin}\otimes u^{\nlin}
			.
		\end{equation*}
		Thus, by equations \eqref{e.w07017}, we get (with $j=0,p=\infty, q_1=\infty$)
		\begin{equation*}\label{e.w07018}
			\|\tilde{P}_{N}(u^{\lin}\otimes u^{\lin})\|_{L_{t}^{1}L_{x}^{\infty}([-1/2,1]\times \R^3)}
			\lesssim
			\|u^{\lin}\otimes u^{\lin}\|_{L_{t}^{1}L_{x}^{\infty,\infty}([-1/2,1]\times \R^3)}
			\lesssim
			M^2
			.
		\end{equation*}
	Then, by \eqref{e:bernstein}, \eqref{e.w07017} and \eqref{e.w070111}, we obtain
	\begin{align*}
		\begin{split}
			\|\tilde{P}_{N}(u^{\nlin}\otimes u^{\lin})\|_{L_{t}^{1}L_{x}^{\infty}([-1/2,1]\times \R^3)}
			\lesssim &
			N^{1/2}
			\|u^{\lin}\otimes u^{\nlin}\|_{L_{t}^{1}L_{x}^{6, \infty}([-1/2,1]\times \R^3)}\\
			\lesssim & N^{1/2}\|u^{\lin}\|_{L^2_tL^{\infty}_x}\|u^{\nlin}\|_{L^2_tL_x^{6, \infty}} \lesssim M^3	N^{1/2}.
		\end{split}
	\end{align*}
	Similarly, we get
	\begin{align*}\label{e.w08131}
		\begin{split}
			&\|\tilde{P}_{N}(u^{\lin}\otimes u^{\nlin})\|_{L_{t}^{1}L_{x}^{\infty}([-1/2,1]\times \R^3)}
			\lesssim M^3	N^{1/2}.
		\end{split}
	\end{align*}
 Next, for $u^{\nlin} \otimes u^{\nlin}$, we further decompose it into ``low-low", ``low-high", ``high-low" and ``high-high" :
	\begin{equation}\label{lin_nlin}
	u^{\nlin} \otimes u^{\nlin} = \Pi_{l-l}+\Pi_{l-h} +\Pi_{h-l} + \Pi_{h-h},
	\end{equation}
where 
		\begin{align*}
			\begin{split}
			\Pi_{l-l} = 
			P_{\leq N}u^{\nlin} \otimes P_{\leq N}u^{\nlin} \comma
			\Pi_{l-h} = 
			P_{\leq N}u^{\nlin} \otimes P_{> N}u^{\nlin} \comma\\
			\Pi_{h-l} = 
			P_{> N}u^{\nlin} \otimes P_{\leq N}u^{\nlin} \comma
			\Pi_{h-h} = 
			P_{> N}u^{\nlin} \otimes P_{> N}u^{\nlin}.
			\end{split}
		\end{align*}
By H\"older's inequality, we have
		\begin{align}\label{eq_four_term}
			\begin{split}
				&\|\Pi_{l-l}\|_{L_{t}^{1}L_{x}^{\infty}([-1/2,1]\times \R^3)}
				\lesssim
				\|P_{\leq N}u^{\nlin}\|^{2}_{L_{t}^{2}L_{x}^{\infty}([-1/2,1]\times \R^3)}
				,\\
				&\|\Pi_{l-h} + \Pi_{h-l}\|_{L_{t}^{1}L_{x}^{2}([-1/2,1]\times \R^3)}
				\lesssim
				\|P_{\leq N}u^{\nlin}\|_{L_{t}^{2}L_{x}^{\infty}([-1/2,1]\times \R^3)} 
				\|P_{> N}u^{\nlin}\|_{L_{t}^{2}L_{x}^{2}([-1/2,1]\times \R^3)}
				,\\
&\|\Pi_{h-h}\|_{L_{t}^{1}L_{x}^{1}([-1/2,1]\times \R^3)}
				\lesssim
				\|P_{> N}u^{\nlin}\|^{2}_{L_{t}^{2}L_{x}^{2}([-1/2,1]\times \R^3)}
				.
			\end{split}
		\end{align}
		Hence, by Bernstein, triangle inequality, Young's inequality,  equations \eqref{lin_nlin}, and \eqref{eq_four_term}, we get
		\begin{equation*}\label{e.w07023}
			\begin{split}
			&	\|\tilde{P}_{N}(u^{\nlin}  \otimes u^{\nlin})\|_{L_{t}^{1}L_{x}^{\infty}([-1/2,1]\times \R^3)} \\
  \lesssim &
				\|P_{\leq N}u^{\nlin}\|^{2}_{L_{t}^{2}L_{x}^{\infty}([-1/2,1]\times \R^3)}
				+
			 	N^{3/2}\|P_{\leq N}u^{\nlin}\|_{L_{t}^{2}L_{x}^{\infty}([-1/2,1]\times \R^3)}
				\|P_{> N}u^{\nlin}\|_{L_{t}^{2}L_{x}^{2}([-1/2,1]\times \R^3)}
				\\& + N^3
				\|P_{> N}u^{\nlin}\|^{2}_{L_{t}^{2}L_{x}^{2}([-1/2,1]\times \R^3)}
				\\ \lesssim &
				\|P_{\leq N}u^{\nlin}\|^{2}_{L_{t}^{2}L_{x}^{\infty}([-1/2,1]\times \R^3)}
				+
				N^{3}\|P_{> N}u^{\nlin}\|^{2}_{L_{t}^{2}L_{x}^{2}([-1/2,1]\times \R^3)}
				.
			\end{split}
		\end{equation*}
		Thus, combining all the above estimates and recalling equations \eqref{e.w06303} yield
		\begin{equation*}
			\begin{split}
				\|P_{N}u^{\nlin}\|_{L_{t}^{1}L_{x}^{\infty}([0,1]\times \R^3)}
				&\lesssim
				M^3N^{-1/2}
				\\&\quad+
				N^{-1}\left( \|P_{\leq N}u^{\nlin}\|^{2}_{L_{t}^{2}L_{x}^{\infty}([-1/2,1]\times \R^3)}
				+
				N^{3}\|P_{> N}u^{\nlin}\|^{2}_{L_{t}^{2}L_{x}^{2}([-1/2,1]\times \R^3)}\right)
				.
			\end{split}
		\end{equation*}
	Then, by equation \eqref{e:bernstein} and Cauchy-Schwarz, we get
	\begin{align*}\
		\begin{split}
			\|P_{\leq N}u^{\nlin}\|_{L_{t}^{2}L_{x}^{\infty}([-1/2,1]\times \R^3)}^{2}
			&=
			\|\sum_{N'\leq N} P_{ N'}u^{\nlin}\|_{L_{t}^{2}L_{x}^{\infty}([-1/2,1]\times \R^3)}^{2}
			\\&\lesssim
			\left(\sum_{N'\leq N}N'^{3/2} \|P_{ N'}u^{\nlin}\|_{L_{t}^{2}L_{x}^{2}([-1/2,1]\times \R^3)}\right)^2
			\\&\lesssim
			\sum_{N'\leq N}N'^{1/2}
			\sum_{N'\leq N}(N')^{5/2} \|P_{ N'}u^{\nlin}\|_{L_{t}^{2}L_{x}^{2}([-1/2,1]\times \R^3)}^2
			\\&	\lesssim 
			N^{1/2}
			\sum_{N'\leq N}N'^{5/2} \|P_{ N'}u^{\nlin}\|^2_{L_{t}^{2}L_{x}^{2}([-1/2,1]\times \R^3)},
		\end{split}
	\end{align*}
	where $N'$ ranges over powers of two. Next, by Plancherel, we see
	\begin{equation*}\label{e:w08062}
		\begin{split}
			\|P_{> N}u^{\nlin}\|_{L_{t}^{2}L_{x}^{2}([-1/2,1]\times \R^3)}^{2}
			=
			\|\sum_{N'> N} P_{ N'}u^{\nlin}\|_{L_{t}^{2}L_{x}^{\infty}([-1/2,1]\times \R^3)}^{2}
			&\lesssim
			\sum_{N'> N}\|P_{ N'}u^{\nlin}\|_{L_{t}^{2}L_{x}^{2}([-1/2,1]\times \R^3)}^2 .	\end{split}
	\end{equation*}
	Thus, after summing in $N$ and applying the triangle inequality together with equation \eqref{e.w070110}, we have
	\begin{align*}
		\begin{split}
			\|P_{\geq 1}u^{\nlin}\|_{L_{t}^{1}L_{x}^{\infty}([0,1]\times \R^3)}
			&\lesssim
			M^{3}\sum_{N}N^{-1/2}
			+
			\sum_{N}
			N^{1/2}
			\sum_{N'\leq N}N'^{5/2} \|P_{ N'}u^{\nlin}\|^2_{L_{t}^{2}L_{x}^{2}([-1/2,1]\times \R^3)}
			\\&\quad+
			\sum_{N}
			N^{2}
			\sum_{N'\leq N}N'^{5/2} \|P_{ N'}u^{\nlin}\|^2_{L_{t}^{2}L_{x}^{2}([-1/2,1]\times \R^3)}
			\\&\lesssim
			M^{3}
			+
			\sum_{N'\leq N}N'^{2} \|P_{ N'}u^{\nlin}\|^2_{L_{t}^{2}L_{x}^{2}([-1/2,1]\times \R^3)}
			\lesssim
			M^{4}
			.
		\end{split}
	\end{align*}
		By \eqref{eq:linbounds} and \eqref{e:bernstein}, 
		We have
		\begin{equation*}\label{e:w08064}
			\begin{split}
				\|u^{\lin}\|_{L_{t}^{1}L_{x}^{\infty}([0,1]\times \R^3)}
				+
				\|P_{<1}u^{\nlin}\|_{L_{t}^{1}L_{x}^{\infty}([0,1]\times \R^3)}
				&\lesssim
				M
				,
			\end{split}
		\end{equation*}
		and thus
		\begin{equation*}\label{e:w08065}
			\begin{split}
				\|u\|_{L_{t}^{1}L_{x}^{\infty}([0,1]\times \R^3)}
				&\lesssim
				M^4
				.
			\end{split}
		\end{equation*}
		This concludes the proof.
	\end{proof}

		\begin{proposition}\label{p.w08291}
			(Epochs of regularity) Let $u , b :[t_0-T, t_0] \times \mathbb{R}^3 \to \mathbb{R}^3$, $p: [t_0-T, t_0]\times \mathbb{R}^3 \to \mathbb{R}$ be a classical solution to the Navier-Stokes equation satisfying \eqref{e:ub-bounds}. Then for any interval $I$ in $[t_0-T/2, t_0]$, there is a subinterval $I' \subset I$ with $|I'| \gtrsim M^{-8}|I| $ such that 
			\begin{equation*}\label{e.w08171}
				\|\nabla^iu\|_{L^{\infty}_t L_x^{\infty} (I' \times \mathbb{R}^3)}   \lesssim M^{O(1)} |I|^{-(i+1)/2}
			\end{equation*}
			and 
			\begin{equation*}	\|\nabla^i\omega\|_{L^{\infty}_t L_x^{\infty} (I' \times \mathbb{R}^3)} \lesssim M^{O(1)} |I|^{-(i+2)/2}
			\end{equation*}
			for $i=0, 1$. 
		\end{proposition}
		\begin{proof}
By rescaling and time translation, we may assume without loss of generality that $I=[0, 1]$ and $[-1, 1] \subset [t_0-T, t_0]$. 	We define the enstrophy-type quantity
\begin{equation*}
\mathcal{E}(t) = \frac{1}{2} \int_{\mathbb R^3} |\nabla u^{\nlin} (t,x)|^2 dx
\end{equation*}
where $\nabla u^{\nlin}$ satisfying equation 
\begin{align}\label{eq : u_nlin2}
\partial_t u^{\nlin}	-	\Delta u^{\nlin}+	u\cdot \nabla u	+	\nabla p	= 0.
\end{align}		
For $\mathcal{E}(t)$,  we are able to find $t_1 \in [0, \frac{1}{2}]$ such that $\mathcal{E}(t_1) \lesssim M^4$.  For $t \in [t_1, t_1 +CM^{-8}]=[\tau(0), \tau(1)]$, where $\tau(s):= t_1 +scM^{-8}$ and small $c>0$,  a continuity argument yields
\begin{equation*}
	\int_{\tau(0)}^{\tau(1)} \int_{\R^3} |\nabla^2  u^{\nlin}|^{2} \,dx \,dt
	\lesssim M^{4}.
\end{equation*}
(more details about this can be found in \cite{tao} on page 13).
By fundamental theorem of calculus, we have  
			\begin{align}\label{e.w08141}
				\|\nabla u\|_{L^{\infty}_{t}L^{2}_{x} ([\tau(0),\tau(1)]\times\R^3)} 
				+
				\|\nabla^{2} u\|_{L^{2}_{t}L^{2}_{x} ([\tau(0),\tau(1)]\times\R^3)} 
				\lesssim
				M^2.
			\end{align}
			By the Gagliardo-Nirenberg inequality, we obtain
			\begin{equation*}
				\label{eq:gagliardo}
				\| u^{\nlin}\|_{L^{\infty}_{x} }
				\lesssim
				\| \nabla u^{\nlin}\|_{L^{2}_{x}}^{1/2} 
				\| \nabla^{2}u^{\nlin}\|_{L^{2}_{x}}^{1/2}
				.
			\end{equation*}
			In particular, we have
			\begin{align*}
				\| u^{\nlin}\|_{L^{4}_{t}L^{\infty}_{x} ([\tau(0),\tau(1)]\times\R^3)}
				\lesssim
				M^2
				\comma
				\| u^{}\|_{L^{4}_{t}L^{\infty}_{x} ([\tau(0),\tau(1)]\times\R^3)}
				\lesssim
				M^2.
			\end{align*}
			Hence, by equations \eqref{e.w07017} and \eqref{e.w08141} and Sobolev embedding, we get	
			\begin{equation}\label{e.w0911}
				\| \nabla u^{\nlin}\|_{L^{2}_ {t}L^{6}_{x} ([\tau(0),\tau(1)]\times\R^3)}
				\lesssim
				M^2
				\comma
				\| \nabla u^{}\|_{L^{2}_{t}L^{6}_{x} ([\tau(0),\tau(1)]\times\R^3)}
				\lesssim
				M^2
				.
			\end{equation}
			Now we do iteration to obtain higher regularity estimates. Assuming $t \in [\tau(0.1),\tau(1)]$, we have
			\begin{align}\label{e.w08142}
				u(t)= e^{(t-\tau(0))\Delta} u(\tau(0))
				-
				\int_{\tau(0)}^{t} e^{(t-s)\Delta} \mathbb P\nabla \cdot (u\otimes u)\,ds
				.
			\end{align}				
			Notice that by heat kernel estimates, we see
			\begin{align*}
				\|e^{(t-s)\Delta}P\nabla \cdot (u\otimes u)\|_{L_{x}^{\infty}}
				\lesssim
				(t-s)^{-1/2}
				\| u\otimes u\|_{L_{x}^{\infty}}
				,
			\end{align*}
			and by \eqref{e:bernsteintime},
			\begin{align*}
				\|e^{(t-\tau(0))\Delta} u(\tau(0))\|_{L_{x}^{\infty}}
				\lesssim
				(t-\tau(0))^{-1/2}
				\| u(\tau(0)\|_{L_{x}^{3, \mathfrak{q_0}}}
				\lesssim
				M^5
				.
			\end{align*}
			Since 
			\begin{equation*}
				(\tau(0.1)-\tau(0))^{-\frac{1}{2}} \leq (t-\tau(0))^{-\frac{1}{2}} \leq (\tau(1)-\tau(0))^{-\frac{1}{2}} \lesssim M^4.
			\end{equation*}
			Therefore, by triangle inequality
			\begin{align*}
				\|u(t)\|_{L^{\infty}_{x}(\R^3)}
				&\leq
				\|e^{(t-\tau(0))\Delta} u(\tau(0))\|_{L^{\infty}_{x}(\R^3)}
				+
				\int_{\tau(0)}^{t} \|e^{(t-s)\Delta} \mathbb P\nabla \cdot (u\otimes u)\|_{L^{\infty}_{x}(\R^3)}\,ds
				\\&\lesssim
				M^5
				+
				\int_{\tau(0)}^{t} (t-s)^{-1/2} \|u(s)\|_{L_{x}^{\infty}(\R^3)}^2 \,ds
				.
			\end{align*}	
			Then, by Young's convolution inequality, we get
			\begin{align*}
				\|u\|_{L_{t}^{8}L^{\infty}_{x} ([\tau(0.1),\tau(1)]\times\R^3)}
				\lesssim
				M^5+\|t^{-\frac{1}{2}}\|_{L_t^{8/7}}\|u\|_{L_t^4L_x^{\infty}}^2\lesssim M^5+M^4\lesssim M^5
				.
			\end{align*}	
			Repeating the above process for $t \in [\tau(0.2),\tau(1)]$ yields
			\begin{align*}
				\|u(t)\|_{L^{\infty}_{x}(\R^3)}
				\lesssim
				M^5
				+
				\int_{\tau(0.1)}^{t} (t-s)^{-1/2} \|u(s)\|_{L_{x}^{\infty}(\R^3)}^2 \,ds
				.
			\end{align*}	
			And thus by H\"older's inequality, we have
			\begin{align}
				\label{eq:inftimespace}
				\|u\|_{L_{t}^{\infty}L^{\infty}_{x} ([\tau(0.2),\tau(1)]\times\R^3)}
				\lesssim
				M^5
				.
			\end{align}	
			Next, we want to show that
			\begin{align*}
				\|\nabla u\|_{L_{t}^{\infty}L^{\infty}_{x} ([\tau(0.4),\tau(1)]\times\R^3)}
				\lesssim
				M^{O(1)}
				.
			\end{align*}	
			Indeed, from the mild formulation, we differentiate both sides of equation \eqref{e.w08142} and get for $t \in [\tau(0.3),\tau(1)]$ 
			\begin{equation*}\label{e.w08143}
				\nabla u(t)= \nabla e^{(t-\tau(0.2))\Delta} u(\tau(0.2))
				-
				\int_{\tau(0.2)}^{t} \nabla e^{(t-s)\Delta} \mathbb P\nabla \cdot (u\otimes u)\,ds
				.
			\end{equation*}				
	Thus we claim that 
	\begin{align*}
		\|\nabla u(t)\|_{L_x^{\infty}(\mathbb{R}^3)}&\lesssim 	\|\nabla e^{(t-\tau(0.2))\Delta} u(\tau(0.2))\|_{L_x^{\infty}(\mathbb{R}^3)}+	\int_{\tau(0.2)}^{t}\| \nabla e^{(t-s)\Delta} \mathbb P\nabla \cdot (u\otimes u)\,ds\|_{L_x^{\infty}(\mathbb{R}^3)}\\ &\lesssim M^{O(1)}+\int_{\tau(0.2)}^t(t-s)^{-3/4}\|\nabla\cdot (u\otimes u)(s)\|_{L_x^{6}(\mathbb{R}^3)}~ds.
	\end{align*}
	Indeed, 
	by \eqref{e:bernsteintime} ($j=1$, $d=3$, $p_2=3$, $ \mathfrak{q}_0 >3$), and $\tau(0.3)=t_1+0.3cM^{-8}$, $\tau(0.2)=t_1+0.2cM^{-8}$, 
	\begin{align*}
		\|\nabla e^{(t-\tau(0.2))\Delta} u(\tau(0.2))\|_{L_x^{\infty}(\mathbb{R}^3)} \leq (t-\tau(0.2))^{-1} \|u(\tau(0.2))\|_{L^{3,\mathfrak{q}_0 }}(\mathbb{R}^3) \lesssim M^{O(1)}.
	\end{align*}
	By \eqref{e:bernsteintime} ($j=1$, $d=3$, $p_2=6$, $q_2=6$),
	\begin{align*}
		\int_{\tau(0.2)}^{t} 	\| \nabla e^{(t-s)\Delta} \mathbb P\nabla \cdot (u\otimes u)\,\|_{L_x^{\infty}(\mathbb{R}^3)} ~ds & \leq \int_{\tau(0.2)}^{t} (t-s)^{-3/4}\|\nabla \cdot (u\otimes u)(s)\|_{L_x^{6}(\mathbb{R}^3)}~ds.\\
	\end{align*}
	By \eqref{e.w0911}, \eqref{eq:inftimespace}, Young's convolution inequality in time $t$, and H\"older's inequality in space $x$, 
	\begin{align*}
		&\| \int_{\tau(0.2)}^{t} (t-s)^{-3/4} 
		\|\nabla \cdot (u\otimes u)(s)\|_{L_x^{6}(\mathbb{R}^3)}~ds\|_{L_t^4([\tau(0.3), \tau(1)])}\\
		&\leq 	\|t^{-3/4}\|_{L_t^{4/3}} \|\nabla \cdot (u\otimes u)\|_{L_t^2 L_x^{6} ([\tau(0.2), \tau(1) ]  \times \mathbb{R}^3)}\\
		&\lesssim \|u \cdot \nabla u\|_{L_t^2 L_x^{6} ([\tau(0.2), \tau(1) ]  \times \mathbb{R}^3)} \\
		&\lesssim \|\nabla u\|_{L_t^2L_x^6([\tau(0), \tau(1)]\times \mathbb{R}^3)} \|u\|_{L_t^{\infty}L_x^{\infty}([\tau(0.2), \tau(1)]\times \mathbb{R}^3)}  \lesssim M^{O(1)}.
	\end{align*}
	Hence, one has
	\begin{equation}
		\label{eq:t4xinfty}
		\|\nabla u\|_{L_t^4L_x^{\infty}([\tau(0.3), \tau(1)] \times \mathbb{R}^3)} \lesssim M^{O(1)}.
	\end{equation}
	By \eqref{eq:inftimespace}, \eqref{eq:t4xinfty}, Leibniz and  H\"older's inequality,
	\begin{equation*}
		\|\nabla\cdot ( u\otimes u)\|_{L_t^4L_x^{\infty}([\tau(0.3), \tau(1)] \times \mathbb{R}^3)} \lesssim M^{O(1)}.
	\end{equation*}
	Similarly by \eqref{e:bernsteintime}($j=1$, $p_1=q_1=p_2=q_2=\infty$), for $t \in [\tau(0.4), \tau(1)]$,
	\begin{align*}
		\|\nabla u(t)\|_{L_x^{\infty}(\mathbb{R}^3)}
		&\leq  \|\nabla e^{(t-\tau(0.3))\Delta} u(\tau(0.3))\|_{L_x^{\infty}(\mathbb{R}^3)}
		+
		\|\int_{\tau(0.3)}^{t} \nabla e^{(t-s)\Delta} \mathbb P\nabla \cdot (u\otimes u)\,ds \|_{L_x^{\infty}(\mathbb{R}^3)}\\
		&\lesssim M^{O(1)}+\int_{\tau(0.3)}^t(t-s)^{-1/2}\|\nabla\cdot (u\otimes u)(s)\|_{L_x^{\infty}(\mathbb{R}^3)}~ds.
	\end{align*}
	Then, by Young's convolution inequality in time $t$ and H\"older's inequality, 
	\begin{align*}
		\|\nabla u\|_{L_t^{\infty}L_x^{\infty}([\tau(0.4), \tau(1)] \times \mathbb{R}^3)} &\lesssim M^{O(1)}.
	\end{align*}
	By the vorticity equation, 
	\begin{equation*}
		\label{eq:vorticity}
		\partial_t\omega=\Delta \omega-(u\cdot \nabla )\omega+(\omega \cdot \nabla)u, 
	\end{equation*}
	we have, 
	\begin{equation*}
		\label{eq:vorticity1}
		\partial_t\omega=\Delta \omega+O(M^{O(1)})(|\omega|+|\nabla \omega|) 
	\end{equation*}
	on $[\tau(0.4), \tau(1)] \times \mathbb{R}^3$ and $\omega=O(M^{O(1)})$ on this slab. By standard parabolic estimates, we obtain 
	\begin{equation*}
		\|\nabla \omega \|_{L_t^{\infty}L_x^{\infty}([\tau(0.5), \tau(1)] \times \mathbb{R}^3)}\lesssim M^{O(1)}. 
	\end{equation*}
	Setting $I'=[\tau(0.5), \tau(1)]$, we obtain the desired conclusion.
	\end{proof}
		\begin{proposition}\label{p.w08292}
			(Back propagation) Let $u :[t_0-T, t_0] \times \mathbb{R}^3 \to \mathbb{R}^3$, $p: [t_0-T, t_0]\times \mathbb{R}^3 \to \mathbb{R}$ be a classical solution to Navier-Stokes that obeys \eqref{e:ub-bounds} and
			let $(t_{1},x_{1}) \in [t_{0}-T/2, t_{0}]\times \R^3$ and $N_1 \geq A_3 T^{-1/2}$ be such that
			\begin{equation*}
				|P_{N_1}u(t_1,x_1)|\geq M_{1}^{-1}N_{1}.
			\end{equation*}
			Then there exists $(t_{2},x_{2}) \in [t_{0}-T/2, t_{1}]\times \R^3$ and $N_2 \in [M_{2}^{-1}N_1 ,M_{2}N_1]$ such that
			\begin{equation*}\label{e.w08151}	
				M_{3}^{-1}N_{1}^{-2}\leq t_1 -t_2\leq M_{3}N_{1}^{-2}
			\end{equation*}
			and
			\begin{equation*}\label{e.w08152}	
				|x_2 -x_1|\leq M_{3}N_{1}^{-1}
			\end{equation*}
			and
			\begin{equation*}\label{e.w08153}	
				|P_{N_2}u(t_2,x_2)|\geq M_{1}^{-1}N_2
				.
			\end{equation*}
		\end{proposition}
		\begin{proof}
			Following \cite{tao}, we renormalize $N_1 =1$ and choose $t_1=0$ so that $t_0 - T \leq -T/2 \leq -M_{3}^{2}/2$. In particular, $[-2M_3 ,0] \subset [t_0 -T,t_0]$. Then by our assumption in equation \eqref{e.w08151}, we see
			\begin{equation*}\label{eq-renom-point}
				|P_{1}u(0,x_{1})|\geq M_{1}^{-1}.
			\end{equation*}
			We now prove the claim by contradiction, i.e., we assume
			\begin{equation*}\label{e.w08154}	
				\|P_{N}u \|_{L_{t}^{\infty}L_{x}^{\infty}([-M_3, -M_{3}^{-1}]\times B(x_{1}, M_4))}
					\leq 
					M_{1}^{-1}N
			\end{equation*}
			for all $M_{2}^{-1}\leq N \leq M_2$. Then, by fundamental theorem of calculus in time, we enlarge the time interval so that
			\begin{equation*}\label{e.w08155}	
				\|P_{N}u\|_{L_{t}^{\infty}L_{x}^{\infty}([-M_3, 0]\times B(x_{1}, M_4))}
					\leq 
					M_{1}^{-1}N
					.
			\end{equation*}
			{\bf Step 1:} Suppose $N\geq M_{2}^{-1}$, then by Duhamel, we get
			\begin{equation*}\label{e.w08156}	
				\begin{split}
					&\|P_{N}u(t)\|_{L_{x}^{3/2,  \mathfrak{q}_0 /2} (B(x_1, M_4))}
					\leq 
					\|e^{(t+2M_{3})\Delta}P_{N}u(-2M_3)\|_{L_{x}^{3/2, \mathfrak{q}_0 /2} (B(x_{1}, M_4))}
					\\&\quad+
					\int_{-2M_3}^{t}\|e^{(t-s)\Delta}P_{N}\mathbb P\nabla \cdot(u\otimes u)\|_{L_{x}^{3/2, \mathfrak{q}_0 /2} (B(x_{1}, M_4))}\,ds
					=I_1+I_2
					,
				\end{split}
			\end{equation*}
			where $t \in [-M_3,0]$. 
			By \eqref{e:bersteinpn} with $p_1=p_2=3$, $q_1=q_2=\mathfrak{q}_0 $, $j=0$, and $-2M_3\leq -t-2M_3\leq -M_3$ and \eqref{e:ub-bounds}, we see
			\begin{align*}
			I_1 &= 	\|e^{(t+2M_{3})\Delta}P_{N}u(-2M_3)\|_{L_{x}^{3/2,  \mathfrak{q}_0 /2} (B(x_{1}, M_4))}  \\
			 & \lesssim  \|\mathds{1}_{B(x_{1}, M_4)}\|_{L_x^{3, \mathfrak{q}_0 }(\mathbb{R}^3)}\|e^{(t+2M_{3})\Delta}P_{N}u(-2M_3)\|_{L_x^{3, \mathfrak{q}_0 }(\mathbb{R}^3)} \\
				&\lesssim M_4e^{-\frac{N^2(t+2M_3)}{20}} \|u(-2M_3)\|_{L_x^{3,\mathfrak{q}_0 }(\mathbb{R}^3)} \lesssim MM_4 e^{-\frac{N^2M_3}{20}}.
			\end{align*}
			By \eqref{e:bersteinpn} with $p_1=p_2=3/2$, $q_1=q_2=\mathfrak{q}_0 /2$, and $j=1$, we obtain
			\begin{align*}
				I_2
				&=\int_{-2M_3}^{t}\|e^{(t-s)\Delta}P_{N}\mathbb P\nabla \cdot(u\otimes u)\|_{L_{x}^{3/2, \mathfrak{q}_0/2} (\mathbb{R}^3)}\,ds 
				\lesssim \int_{-2M_3}^{t} Ne^{-\frac{N^2(t-s)}{20}}\|(u\otimes u)\|_{L_{x}^{3/2, \mathfrak{q}_0/2} (\mathbb{R}^3)}\,ds  
				\\&\lesssim \int_{-2M_3}^{t} Ne^{-\frac{N^2(t-s)}{20}}\|u\|^2_{L_{x}^{3, \mathfrak{q}_0}(\mathbb{R}^3)} \,ds
				\lesssim M^2N^{-1} (1-e^{-\frac{N^2(t+2M_3)}{20}}) \lesssim M^2N^{-1}.
			\end{align*}
			Thus, combining the estimates of $I_1$ and $I_2$ above yields
			\begin{equation*}\label{e.w08281}	
				\begin{split}
					\|P_{N}u(t)\|_{L_{t}^{\infty}L_{x}^{3/2, \mathfrak{q}_0/2} (B(x_1, M_4))}
					&\leq
					MM_4 e^{-\frac{N^2M_3}{20}}+ M^{2}N^{-1} \leq M^{2}N^{-1}
					.
				\end{split}
			\end{equation*}
			Hence  in the range $N\geq M_2^{-1}$,
			\begin{equation}\label{e.f10291}	
				\|P_{N}u(t)\|_{L_{t}^{\infty}L_{x}^{3/2, \mathfrak{q}_0/2} ([-M_3, 0]\times B(x_1, M_4))}
				\lesssim M^{2}N^{-1}.
			\end{equation}
			{\bf Step 2:} Suppose $N\geq M_{2}^{-1/2}$, then by Duhamel, we get
			\begin{equation*}\label{e.w08162}	
				\begin{split}
					\|P_{N}u(t)\|_{L_{x}^{1, \mathfrak{q}_0/2} (B(x_1, M_4 /2))}
					&\leq 
					\|e^{(t+2M_{3})\Delta}P_{N}u(-2M_3)\|_{L_{x}^{1, \mathfrak{q}_0/2} (B(x_{1}, M_4 /2))}
					\\&\quad+
					\int_{-M_3}^{t}\|e^{(t-s)\Delta}P_{N}\mathbb P\nabla \cdot(u\otimes u)\|_{L_{x}^{1, \mathfrak{q}_0/2} (B(x_{1}, M_4 /2))}\,ds
					,
				\end{split}
			\end{equation*}
			where $t \in [-M_3/2,0]$. We apply H\"older's inequality, \eqref{e:ub-bounds}, and \eqref{e:bersteinpn} and obtain
			\begin{equation}\label{e.w08282}	
				\begin{split}
					\|e^{(t+2M_{3})\Delta}P_{N}u(-2M_3)\|_{L_{x}^{1, \mathfrak{q}_0/2} (B(x_{1}, M_4 /2))}
					\lesssim
					MM_{4}^{2} e^{-N^2 M_3 /40}
					.
				\end{split}
			\end{equation}
          Then, we apply our multiplier theorem and obtain for the range $N\geq M_2^{-1/2}$, 
			\begin{equation}
				\label{e.f10293}	
				\|P_{N}u(t)\|_{L_{t}^{\infty}L_{x}^{1, \mathfrak{q}_0/2} ([-{M_3}/2, 0]\times B(x_1, {M_4}/2))}
				\lesssim M^{3}N^{-2}.
			\end{equation}		
			\\
			{\bf Step 3:} Suppose $M_{2}^{-1/3}\leq N\leq  M_{2}^{1/3}$.
			\begin{equation*}\label{e.w08163}	
				\begin{split}
					\|P_{N}u(t)\|_{L_{x}^{2, \mathfrak{q}_0/2} (B(x_1, M_4/4))}
					&\leq 
					\|e^{(t+2M_{3})\Delta}P_{N}u(-M_3/2)\|_{L_{x}^{2, \mathfrak{q}_0/2} (B(x_{1}, M_4/4))}
					\\&\quad+
					\int_{-2M_3}^{t}\|e^{(t-s)\Delta}P_{N}\mathbb P\nabla \cdot(u\otimes u)\|_{L_{x}^{2, \mathfrak{q}_0/2} (B(x_{1}, M_4/4))}\,ds
					,
				\end{split}
			\end{equation*}
			where $t \in [-M_3/3,0]$. 		
			Similar to equation \eqref{e.w08282} and by \eqref{e:multiplier2} ($p_1=2, p_2=1, q_1=q_2=\frac{\mathfrak{q}_0}{2}$), we obtain
			\begin{equation*}
				\label{e.f10292}
				\|P_Nu\|_{L_t^{\infty} L_x^{2, \mathfrak{q}_0/2} ([-M_3/4, 0] \times B(x_1, M_4/4))} \leq M_4^{-40} +N^{1/2} \|\widetilde{P}_N (u(t')\otimes u(t'))\|_{L_t^{\infty}L_x^{1, \mathfrak{q}_0/2}[-M_3/2, 0] \times B(x_1, M_4/3)}.
			\end{equation*}
			We split $\tilde{P}_N(u(t')\otimes u(t'))$ into ``low-high", ``high-low", and ``high-high" terms, that is, 
				\begin{equation*}
					\tilde{P}_N(u\otimes u)=	\pi_{h-l} + \pi_{l-h} + \pi_{h-h},
				\end{equation*}
				where 
				$$
				\pi_{h-l} =  \sum_{N_1 \sim N} \tilde{P}_N(P_{N_1}u \otimes P_{\leq {N_1}/100}u), \quad
				\pi_{l-h} = \sum_{N_1 \sim N}
				\tilde{P}_N ( P_{\leq {N_1}/100} \otimes P_{N_1}u )$$
				$$
				\pi_{h-h} =	\tilde{P}_N \sum_{N_1\sim N_2\geq N} P_{N_1}u \otimes P_{N_2}u.
				$$
				Notice that in both $\pi_{h-l}$ and $\pi_{l-h}$, we have $O(1)$ terms of the ``high-low" form or the ``low-high" form, so we only need to treat the term inside the sum. We do here the estimate for $\pi_{h-l}$, the rest follows similarly. By triangle inequality, \eqref{e.f10291} and pointwise derivative estimate \eqref{e:pnub-bounds}, we have
			\begin{align*}
				&\| P_{\leq N/100} u \|_{L_t^{\infty}L_x^{3/2, \mathfrak{q}_0/2}([-M_3, 0]\times B(x_1, M_4))} 
				\leq \|\sum_{k=0}^{\infty} P_{2^{-k} N/100} u\|_{L_t^{\infty}L_x^{3/2, \mathfrak{q}_0/2}} \\
			 &	\leq  \; \sum_{k=0}^{\infty} \|P_{2^{-k} N/100} u\|_{L_t^{\infty}L_x^{3/2, \mathfrak{q}_0/2}} \\
			 &	\leq \sum_{k=0}^{100} \|P_{2^{-k} N/100} u\|_{L_t^{\infty}L_x^{3/2, \mathfrak{q}_0/2}} +\sum_{k=100}^{\infty} \|P_{2^{-k} N/100 }u\|_{L_t^{\infty}L_x^{3/2, \mathfrak{q}_0/2}} \\
			 & 	\lesssim\sum_{k=0}^{100} M^2\frac{2^k100}{N} + \sum_{k=100}^{\infty}  M \frac{2^{-k}N}{100} \lesssim M^2N^{-1}.
			\end{align*}
			For the high-low term, by \eqref{e:multiplier2} and \eqref{e:bernstein}, we see
			\begin{align*}
				&\|	\tilde{P}_N( P_{N_1}u \otimes P_{\leq {N}/100}u)\|_{L_x^{1, \mathfrak{q}_0/2}B(x_1, {M_4}/3)} \leq \|P_{N_1}u \otimes P_{\leq {N}/100}u\|_{L_x^{1, \mathfrak{q}_0/2}B(x_1, {M_4}/2)} +M_4^{-40} \\
				& \lesssim \|P_{N_1}u\|_{L^{\infty}_x} \|P_{\leq {N}/100}u\|_{L_x^{1, \mathfrak{q}_0/2}B(x_1, {M_4}/2)}+M_4^{-40}\\
				&\lesssim M_1^{-1} N  N^{-1} \|P_{\leq {N}/100}u\|_{L_x^{3/2, \mathfrak{q}_0/2}B(x_1, {M_4}/2)}\lesssim M^3M_1^{-1} N^{-1}. 
			\end{align*}	
			For the ``high-high" term, by \eqref{e.f10293},
			\begin{align*}
				&\left\|\tilde{P}_N\left(\sum_{N_1\sim N_2\geq N} P_{N_1}u \otimes P_{N_2}u\right)\right\|_{L_x^{1, \mathfrak{q}_0/2} B(x_1, {M_4}/3)} 
				\leq \sum_{N_1\sim N_2\geq N} \|P_{N_1}u \otimes P_{N_2}u\|_{L_x^{1, \mathfrak{q}_0/2} }+M_4^{-40} 
				\\&\lesssim \sum_{N_1\sim N_2\geq N} \|P_{N_1}u\|_{L_x^{1,\mathfrak{q}_0/2}} \|P_{N_2}u\|_{L^{\infty}} +M_4^{-40} 
				 \lesssim \sum_{N_1\sim N_2\geq N}  M^3N_1^{-2} M_1^{-1} N_2  
				\\& \lesssim M^3M_1^{-1} \sum_{N_1\geq N} N_1^{-1} \lesssim M^3M_{1}^{-1}N^{-1}.
			\end{align*}
				Gathering the estimates above, we conclude that for frequency $M_2^{-1/3} \leq N \leq M_2^{1/3}$, 
				\begin{equation}\label{eq_step3}
					\|P_N u \|_{L^\infty_t L^{2, \frac{\mathfrak{q}_0}{2}}_x ([- M_3/4, 0]\times B(x_1, M_4/4))} \lesssim M^3 M_1^{-1} N^{-1/2}.
				\end{equation} 
				We see that all the estimates above we obtained in step 1-3 hold in the time interval $[- \frac{M_3}{4}, 0]$, so let us apply Duhamel's formula on this interval to get that
				\begin{align*}
					M_1^{-1} \leq |P_1 u(0, x_1)|
					& \lesssim \Big|e^{\frac{M_3}{4} \Delta} P_1 u (- \frac{M_3}{4})\Big| (x_1) + \int_{- \frac{M_3}{4}}^{0} |e^{-t' \Delta} P_1 \nabla \cdot \tilde{P}_1 (u(t') \otimes u(t'))| (x_1) \, dt' \\
					& \lesssim e^{-\frac{1}{20} \frac{M_3}{4}} M_3 + \int_{- \frac{M_3}{4}}^{0} e^{\frac{t'}{20}} 
					\left(\|\tilde{P}_1 (u(t') \otimes u(t')) (\cdot)\|_{L^{1, \frac{\mathfrak{q}_0}{2}}_x (B(x_1, M_1))} + M_1^{-50}\right)\, dt'
					,
				\end{align*}
				where we used Bernstein's inequality \eqref{e:bersteinpn} ($p_1 = q_1 = \infty$, $p_2 = 3$, $q_2= \mathfrak{q}_0$) and \eqref{e:ub-bounds} for the first term and Bernstein's inequality \eqref{e:bersteinpn} ($p_1 = q_1 = \infty$, $p_2 = 1, q_2 = \frac{\mathfrak{q}_0}{2}$) and local version of multiplier estimate \eqref{e:multiplier2} ($p_1= q_1= \infty, p_2= q_2 = 1, p_3= 3, q_3 = \mathfrak{q}_0$) for the second term. We see that the factor $ e^{-\frac{1}{20} \frac{M_3}{4}} M_3$ on the right-hand side of the inequality above is negligible compared to the integral term, so that by the pigeonhole principle, for some $t' \in [-\frac{M_3}{4}, 0]$, we have 
				$$
				M_1^{-1} \lesssim \|\tilde{P}_1 (u(t') \otimes u(t')) (\cdot) \|_{L^{1, \frac{\mathfrak{q}_0}{2}}_x (B(x_1, M_1))} .
				$$
				We fix this $t'$ and split $\tilde{P}_1 (u(t') \otimes u(t')) $ into three sum terms as in the step 3. For ``high-low" term and ``low-high" term, by local version of the multiplier theorem \eqref{e:multiplier2} and Hölder inequality \eqref{ineq : Holder} we have
				\begin{align*}
					& \|\tilde{P}_1 (P_{N_1} u(t') \otimes P_{\leq 1/100} u(t')) (\cdot) \|_{L^{1, \frac{\mathfrak{q}_0}{2}}_x (B(x_1, M_1))} \\
					\lesssim & \, \|P_{N_1} u(t')\|_{L^{2,  \mathfrak{q}_0}_x (B(x_1, 2 M_1))} \| P_{\leq 1/100} u(t') \|_{L^{2, \mathfrak{q}_0}_x (B(x_1, 2 M_1))} + M_1^{-50}\\
					\lesssim & \, \|P_{N_1} u(t')\|_{L^{2, \frac{\mathfrak{q}_0}{2}}_x (B(x_1, 2 M_1))} \| P_{\leq 1/100} u(t') \|_{L^{2, \frac{\mathfrak{q}_0}{2}}_x (B(x_1, 2 M_1))} + M_1^{-50}\\
					\lesssim & M^3 M_1^{-1} N_1^{-1/2} M^3 M_1^{-1}  + M_1^{-50} \lesssim M^6 M_1^{-2}  
					,
				\end{align*}
				where we used \eqref{eq_step3} (notice that $N \geq M_2^{-1/3}$).
				For the high-high term, 
				\begin{align*}
					& \sum_{N_1 \sim N_2 \gtrsim 1} \|\tilde{P}_1 (P_{N_1} u(t') \otimes P_{N_2} u(t')) (\cdot) \|_{L^{1,\frac{\mathfrak{q}_0}{2}}_x (B(x_1, M_1))} \\
					= & \sum_{1 \lesssim N_1 \sim N_2 \lesssim M_2^{1/3}} \|\tilde{P}_1 (P_{N_1} u(t') \otimes P_{N_2} u(t')) (\cdot) \|_{L^{1, \frac{\mathfrak{q}_0}{2}}_x (B(x_1, M_1))} \\
					& + \sum_{N_1 \sim N_2 \gtrsim M_2^{1/3}} \|\tilde{P}_1 (P_{N_1} u(t') \otimes P_{N_2} u(t')) (\cdot) \|_{L^{1, \frac{\mathfrak{q}_0}{2}}_x (B(x_1, M_1))} \\
					\lesssim &  
					\sum_{1\lesssim N_1 \sim N_2 \lesssim M_2^{1/3}} 
					\|P_{N_1} u(t')\|_{L^{2, \frac{\mathfrak{q}_0}{2}}_x (B(x_1, 2 M_1))} \| P_{N_2} u(t') \|_{L^{2, \frac{\mathfrak{q}_0}{2}}_x (B(x_1, 2 M_1))} \\
					& + 
					\sum_{N_1 \sim N_2 \gtrsim M_2^{1/3}}
					\|P_{N_1} u(t')\|_{L^{3, \mathfrak{q}_0}_x (B(x_1, 2 M_1))} \| P_{N_2} u(t') \|_{L^{3/2, \mathfrak{q}_0}_x (B(x_1, 2 M_1))}\quad\quad\quad(\text{by H\"older's inequality}) 
					\\ \lesssim & \, M^6 M_1^{-2}  + M^3 M_2^{-1/3} \quad\quad\quad\quad\quad\quad(\text{by equations }\eqref{eq_step3}\text{ and }\eqref{e.f10291}) 
					\\\lesssim&  M^6 M_1^{-2}.
				\end{align*}
				 Gathering all the estimates above we obtain 
				$$
				M_1^{-1} \lesssim  M^6 M_1^{-2},
				$$
				which gives a contradiction.
		\end{proof}
		\begin{proposition}\label{p.w08293}
			(Iterated back propagation, \cite{tao}) Let $x \in \R^3$ and $N_0 > 0$ be such that
			\begin{equation*}
				|P_{N_0}u(t_1,x_1)|\geq M_{1}^{-1}N_{0}.
			\end{equation*}
			Then for every $M_{4} N_{0}^{-2}\leq T_1 \leq M_{4}^{-1}T$, there exists	
				$(t_1 , x_1) \in [t_0 - T, t_0 -M_{3}^{-1}T_1] \times \R^3$ and  $ N_1 = M_{3}^{O(1)} T_{1}^{-1/2} $
			such that
			\begin{equation*}	
				x_1 = x_0 +O(M_{4}^{O(1)}T_{1}^{1/2} ) \comma	
				|P_{N_1}u(t_1,x_1)|\geq M_{1}^{-1}N_{1}.
			\end{equation*}
		\end{proposition}
		\begin{proposition}\label{p.w08294}
			(Annuli of regularity) Let $u :[t_0-T, t_0] \times \mathbb{R}^3 \to \mathbb{R}^3$, $p: [t_0-T, t_0]\times \mathbb{R}^3 \to \mathbb{R}$ be a classical solution to Navier-Stokes that obeys \eqref{e:ub-bounds}. If $0<T'<T/2$, $x_0 \in \R^3$, and $R_0 \geq (T')^{1/2}$, then there exists a scale
			\begin{equation*}
				R_0 \leq R\leq \exp{(M_{6}^{O(1)})R_0}
			\end{equation*}
			such that on the region
			\begin{equation*}
				\Omega:= \{(t,x)\in [t_0 -T', t_0]\times \R^3: R\leq |x-x_0|\leq M_6 R \}
			\end{equation*}
			we have 
			\begin{equation*}	\|\nabla^iu\|_{L^{\infty}_t L_x^{\infty} (\Omega)} \lesssim M_{6}^{O(1)} |T'|^{-(i+1)/2}\comma	\|\nabla^i\omega\|_{L^{\infty}_t L_x^{\infty} (\Omega)} \lesssim M_{6}^{O(1)} |T'|^{-(i+1)/2}
			\end{equation*}
			for $i=0,1$.
	\end{proposition}
	Before proving the proposition above, let us give some useful lemmas.	
		\begin{lemma}\label{lemma_pigeonhole}
			Let $A, R_0 >0$ and $A_{6} >>A$. Assume that $\int_{\R^3} f (x) \, dx \leq A$, then we can find a scale  
			$A_{6}^{100} R_0 \leq R \leq \exp(A_{6}^{100} ) R_0$
			such that 
			$$\int_{\mathcal{I}_R} f (x) \,dx \leq A_{6}^{-10}, $$
			where
			$\mathcal{I}_{R} := \{x: A_{6}^{-10} R \leq |x| \leq A_{6}^{10} R\}$.
		\end{lemma}
		\begin{proof}
			The main idea is by the pigeonhole principle, so let us suppose that for every scale $R$, we have $\int_{\mathcal{I}_R} f \,dx > A_{6}^{-10} $, then we give proof by contradiction. Let us construct a sequence of $R_n $ as follows: 
			$$R_1 = A_{6}^{100} R_0, \quad R_2 = A_{6}^{200} R_0,\quad \cdots, \quad R_n = A_{6}^{100 n} R_0 \leq \exp(A_{6}^{100} ) R_0 ,$$
			we thus obtain a sequence of annulus disjoint
			\begin{align*}
				&\mathcal{I}_{R_1} = \{x: A_{6}^{90} R_0 \leq |x| \leq A_{6}^{110} R_0\}, \\
				&\mathcal{I}_{R_2} = \{x: A_{6}^{190} R_0 \leq |x| \leq A_{6}^{210} R_0\}, \\
				& \cdots \\
				&\mathcal{I}_{R_n} = \{x: A_{6}^{100n -10} R_0 \leq |x| \leq A_{6}^{100n +10} R_0\}. 
			\end{align*}
			On one hand, we have $A_{6}^{100 n} R_0 \leq \exp(A_{6}^{100} ) R_0 $, so that $n$ has a upper bound $n \leq \frac{A_{6}^{100}}{100 \ln (A_{6})}$.
			Set $n_0 : = \left\lfloor{\frac{A_{6}^{100}}{100 \ln (A_{6})}}\right\rfloor$ where $\lfloor x \rfloor$ means the integer closest to $x$. 
			Summing the integrals together up to $n_0$, we get that
			$$\sum^{n_0}_{k= 1} \int_{\mathcal{I}_{R_k}} f (x) \, dx > n_0 A_{6}^{-10} > A_{6}. $$
			On the other hand, as these annulus are disjoint, by the assumption $\int_{\R^3} f (x) \, dx \leq A$, we have $$\sum^{n_0}_{k= 1} \int_{\mathcal{I}_{R_k}} f \, dx \leq \int_{\R^3} f (x) \, dx \leq A << A_{6}.$$
			So the lemma is proved.
	\end{proof}
	Notice that in order to prove Proposition \ref{p.w08294}, the estimate that we obtained previously for the linear component (see \eqref{e.w07017}) is not sufficient. The following Lemma is devoted to a precise estimate of the linear component $u^{\lin}$, which is the first step of the proof of Proposition \ref{p.w08294}.
		\begin{lemma}\label{Lemma_linear}
			Assume that $u :[t_0 - T, t_0] \times \mathbb{R}^3 \to \mathbb{R}^3$ is a classical solution to Navier-Stokes that obeys \eqref{e:ub-bounds}. If $0 < T' < \frac{T}{2}$, $x_0 \in \R^3$ and $R_0 \geq \sqrt{T'}$,  
			then there exists a time $t_1 \in [-\frac{t_0}{2}, t_0 -T']$ and a scale 
			\begin{equation*}\label{scale_R}
				M_6^{100} R_0 \leq R \leq \exp{(M_6^{O(1)} R_0)}
			\end{equation*}
			such that 
			\begin{align*}
				\sup_{t_1 \leq t\leq 1}\sup_{\mathcal{I}_{M_6^8 R}} |\nabla^{j}  u^{\lin}(t, x)|^{2} 
				\lesssim
				M_6^{-3} \quad \text{and} \quad
				\sup_{t_1 \leq t\leq 1}\sup_{\mathcal{I}_{M_6^8 R}} |\nabla^{j}  \omega^{\lin}(t, x)|^{2} 
				\lesssim
				M_6^{-3}, \quad \text{for} \quad j=0,1
			\end{align*}
			where $\mathcal{I}_{M_6^{k} R} := \{x: M_6^{-k}R\leq |x| \leq M_6^{k}R\}$ and $k \in \mathbb N$.
		\end{lemma}
		\begin{proof}
			By rescaling, we may assume that $[t_0 - T', t_0] = [0, 1]$. As $0 < T' < \frac{T}{2}$, we have $[-1, 1] \subset [t_0 - T, t_0]$.
			Recalling the previous bound for the linear part  (see \eqref{e.w07017})
			$$\|\nabla^{j}u^{\lin}\|_{L_{t}^{\infty}L_{x}^{p,q_1} ([-1/2, 0] \times \R^3)} \lesssim M, \quad \text{with} \quad j\geq 0, \quad 3\leq p \leq \infty, \quad \frac{1}{q_1}\leq \frac{1}{p}+\frac{1}{\mathfrak{q}_{0}}+\frac{2}{3},$$  
			and by choosing $p= q_1=3$, we can find that there exists a time $t_1 \in [-1/2, 0]$ such that
			\begin{equation*}\label{e.w08172}
				\int_{\R^3}  |\nabla^{j}u^{\lin}(t_1, x)|^3 \, dx 
				\lesssim M^3, \quad \text{for all } \quad j \geq 0. 
			\end{equation*}
			Fixing this $t_1$, by Lemma \ref{lemma_pigeonhole}, we are able to find a scale $M_6^{100} \leq R \leq \exp{(M_6^{O(1)})}$ such that
			\begin{align*}
				\int_{\mathcal{I}_{M_6^{10} R}}
				|\nabla^{j}u^{\lin}(t_1, x)|^3\,dx \lesssim M_6^{-10}, \quad \text{for all} \quad j \geq 0.
			\end{align*}
			In particular, we have
			\begin{align*}
				\|u^{\lin}(t_1, \cdot)\|^3_{W^{1,3} (\mathcal{I}_{M_6^{10} R}) } \leq M_6^{-10} \quad \text{and} \quad
				\|\omega^{\lin}(t_1, \cdot)\|^3_{W^{1,3} (\mathcal{I}_{M_6^{10} R}) } \leq M_6^{-10}
			\end{align*}
			Let us now fix this $R$ and propagate the above estimate to $[t_1,1]$. By Sobolev's inequality, we obtain that 
			\begin{align*}
				\sup_{\mathcal{I}_{M_6^{9} R}} |\nabla^{j}  u^{\lin}(t_1, x)| \lesssim M_{6}^{-3} \quad \text{and} \quad \sup_{\mathcal{I}_{M_6^{9} R}} |\nabla^{j}  \omega^{\lin}(t_1, x)| \lesssim M_{6}^{-3} 
			\end{align*}
			for $j=0,1$.
			As we have $\partial_t \nabla^{j}  u^{\lin} = \Delta \nabla^{j}  u^{\lin}$ and $\partial_t \nabla^{j}  \omega^{\lin} = \Delta \nabla^{j}  \omega^{\lin}$, so we can solve the linear heat equation with initial data at time $t_1$, which implies that
			\begin{align*}
				\sup_{t_1 \leq t\leq 1}\sup_{\mathcal{I}_{M_6^{8} R}} |\nabla^{j}  u^{\lin}(t, x)|\lesssim M_{6}^{-3} \quad \text{and} \quad  \sup_{t_1 \leq t\leq 1}\sup_{\mathcal{I}_{M_6^{8} R}} |\nabla^{j}  u^{\lin}(t, x)|\lesssim M_{6}^{-3}
			\end{align*}
			for $j=0,1$. The lemma is proved.
		\end{proof}
	The following lemma gives an estimate of nonlinear part $u^{\nlin}$ with localization. The strategy is to introduce a time-dependent cut-off function and by energy method, we are able to show that nonlinear part is bounded locally in some proper annuli depending on time. 
		To do this, we first introduce two time-dependent radii
		\begin{equation*}
			R_{-}(t):=
			R_{-}+ C_0 \int_{t_1}^{t}(M_6 + \|u(s)\|_{L_{x}^{\infty}})\,ds, \quad
			R_{+}(t):=
			R_{+}- C_0 \int_{t_1}^{t}(M_6 + \|u(s)\|_{L_{x}^{\infty}})\,ds
		\end{equation*}
		with
		\begin{equation*}
			R_{-}\in [M_{6}^{-8}R, 2M_{6}^{-8}R];
			\quad
			R_{+}\in [M_{6}^{8}R/2, M_{6}^{8}R],
		\end{equation*}
		where $R$ is the same scale in Lemma \ref{Lemma_linear}. With $R_{-}(t)$ and $R_{+}(t)$, we define the following time-dependent cut-off function
		\begin{equation}\label{def_func_theta}
			\theta(x,t):= \max\{\min\{M_6, |x|-R_{-} (t), R_{+}(t)-|x|\},0\}.
		\end{equation}
		Notice that $\theta(t)$ is equal to $M_6$ for $x \in \{R_{-}(t)+M_6\leq |x| \leq R_{+}(t)-M_6\}$ and the support of $\theta$ is the annulus $\{R_{-} (t)\leq |x| \leq R_{+}(t)\}$ with $R_{-} (t) \in [M_{6}^{-8}R, 3M_{6}^{-8}R]$
		and $R_{+} (t) \in [M_{6}^{8}R/3, M_{6}^{8}R]$. 
		Indeed, by the choice of $R$ in Lemma \ref{Lemma_linear}, we have
		\begin{equation*}
			R_{-}(t)\leq R_- +C_0 M_6 (t-t_1) + C_0 M^4 (t-t_1)^{1/2}
			\leq R_- +M_{6}^{-8}R
			\leq 3 M_{6}^{-8}R
			.
		\end{equation*}
		In particular, we have $[R_-(t), R_+(t)] \subset  [M_{6}^{-8}R, M_{6}^{8}R] \subset [M_{6}^{-10}R, M_{6}^{10}R]$ .
		\begin{lemma}\label{Lemma_enstropy}
			For $t\in [t_1,1]$, we have the following estimate for the enstrophy localization
			\begin{equation*}
				\int_{\R^3} |\omega^{\nlin}|^2 \theta(t,x)\,dx \lesssim M_{6}^{-2}
				,
			\end{equation*}
			where $\theta$ is defined in \eqref{def_func_theta}. Moreover, 
			$$
			\int_{t_1}^1 \int_{\R^3} |\nabla \omega^{\textrm{nlin}} (t,x)|^2 \theta(t,x) dx dt \lesssim M_6^{-2}.
			$$
		\end{lemma}
	\begin{proof}
		Let us decompose the vorticity $\omega:=\omega^{\lin}+\omega^{\nlin}=\nabla \times u^{\lin}+\nabla \times u^{\nlin}$. Obviously, the linear part solves the heat equation $ \partial_{t}\omega^{\lin}-\Delta \omega^{\lin}=0 $ and thus the nonlinear parts satisfy
			\begin{equation}\label{eq : vorticity}
				\partial_{t}\omega^{\nlin}-\Delta \omega^{\nlin}=-u\cdot \nabla \omega+\omega\cdot \nabla u.
			\end{equation}
		In order to derive the estimate for $\omega^{\nlin}$, we first recall the previous bound for the nonlinear part of the velocity, $	\int_{-1/2}^{1} \int_{\R^3} |\nabla  u^{\nlin}|^{2} \,dx \,dt \lesssim M^{4}$ (see \eqref{e.w07019}) and using the same time $t_1 \in [-\frac{1}{2}, 0]$ obtained in the proof of Lemma \ref{Lemma_linear}, we have 
			\begin{align*}
				\int_{\R^3} |\nabla  u^{\nlin}(t_1, x)|^{2} \,dx 
				\lesssim M^4.
			\end{align*}
			As $M << M_6$, we get that 
			\begin{align*}
				\int_{M_{6}^{-10}R\leq |x| \leq M_{6}^{10}R} |\nabla  u^{\nlin}(t_1, x)|^{2} \, dx \lesssim M_{6}^{-10},
			\end{align*}
			where $R$ satisfies the same scale $M_{6}^{100}R_0 \leq R \leq \exp{(M_{6}^{O(1)} R_0)}$. To simplify the computation, we define the enstrophy localization by $E(t)$, i.e., 
			$$
			E(t) = \frac{1}{2}\int_{\R^3} |\omega^{\nlin}(t,x)|^2 \theta(t,x)\,dx.
			$$
			Then, from the construction of the cut-off function $\theta(x,t)$, we have
			\begin{align*}
				E(t_1)
				 = \frac{1}{2}\int_{\R^3} |\omega^{\nlin}(t_1,x)|^2 \theta(t_1,x)\,dx 
				\lesssim M_6 \int_{M_{6}^{-10}R \leq |x| \leq M_{6}^{10}R} |\omega^{\nlin}(t_1,x)|^2 \,dx  \lesssim M^{-9}_{6}.
			\end{align*}
		From the vorticity equation \eqref{eq : vorticity} and integrating by part we obtain
			$$
			\partial_t E(t) = - F_1 (t) + F_2 (t)+ F_3 (t) + F_4 (t) + F_5 (t) + F_6 (t),
			$$
			where 
			\begin{align*}
			&	F_1 (t) = \int_{\R^3} |\nabla \omega^{\textrm{nlin}} (t,x)|^2 \theta(t,x) dx, 
				\quad 
				F_2 (t) = - \frac{1}{2}\int_{\R^3} | \omega^{\textrm{nlin}} (t,x)|^2 \partial_t \theta(t,x) dx, \\
			&	F_3 (t) = \frac{1}{2}\int_{\R^3} | \omega^{\textrm{nlin}} (t,x)|^2 \Delta \theta(t,x) dx, 
				\quad 
				F_4 (t) = \frac{1}{2}\int_{\R^3} | \omega^{\textrm{nlin}} (t,x)|^2 u(t,x) \cdot \nabla \theta(t,x) dx, \\
			&	F_5 (t) = -\int_{\R^3} \omega^{\textrm{nlin}} \cdot (u(t,x) \cdot \nabla) \omega^{\textrm{nlin}} \theta(t,x)dx, \quad 
				F_6 (t) = \int_{\R^3} \omega^{\textrm{nlin}} \cdot (\omega^{\textrm{lin}} \cdot \nabla) u^{\textrm{nlin}} \theta(t,x) dx,\\
			&	F_7 (t) = \int_{\R^3} \omega^{\textrm{nlin}} \cdot (\omega^{\textrm{nlin}} \cdot \nabla) u^{\textrm{lin}} \theta(t,x) dx, \quad 
				F_8 (t) = \int_{\R^3} \omega^{\textrm{nlin}} \cdot (\omega^{\textrm{lin}} \cdot \nabla) u^{\textrm{lin}}\theta(t,x) dx,\\
			&	F_9 (t) = \int_{\R^3} \omega^{\textrm{nlin}} \cdot (\omega^{\textrm{lin}} \cdot \nabla) u^{\textrm{nlin}}\theta(t,x) dx.
			\end{align*}
			We notice that 
			\begin{equation*}
				\partial_t \theta(t,x)
				=
				-C_{0}(M_{6}^{-2} + \|u\|_{L^{\infty}(\R^3)})|\nabla \theta(t,x)|
			\end{equation*}
			Thus, $F_{2}\geq 0$.
			Next, the estimates of $F_3(t)$ and $F_4$ follow exactly the same estimates as in \cite{tao} and thus
			\begin{equation}\label{e.w10231}
				F_{4}(t)\leq C_{0}^{-1} F_{2}(t)
				\comma
				\int_{t_1}^{1}|F_{3}(t)|\,dt\leq M_{6}^{-10}
				.
			\end{equation}
			Furthermore, by the hypothesis \eqref{e:ub-bounds} and equation \eqref{e.w07017}, we get
			\begin{equation}\label{e.w10233}
				F_{5}(t)\leq E(t)
				+
				\int_{\R^3} |u\cdot \nabla \omega^{\lin}|^2\theta(t,x)\,dx
				\leq
				E(t)+M_{6}^{-2}
				.
			\end{equation}
			Similarly, we have
			\begin{equation}\label{e.w10234}
				F_{7}(t)\leq E(t)
				\comma
				F_{9}(t)\lesssim E(t)+M_{6}^{-2}
				.
			\end{equation}
			and by Lemma \ref{Lemma_linear}, we see
			\begin{equation}\label{e.w10235}
				F_{8}(t)\leq E(t)+F_{10}(t)
				,
			\end{equation}
			where
			\begin{equation*}
				F_{10}(t) =M_{6}^{-3}\int_{\R^3} |\nabla u^{\lin}|^2 \,dx
			\end{equation*}
			with
			\begin{equation}\label{e.w10236}
				\int_{t_1}^{1}|F_{10}(t)|\,dt \lesssim M_{6}^{-2}
				.
			\end{equation}
			The estimate of $F_{6}$ will be exactly like Tao and we see
			\begin{equation*}
				F_{6}(t)=F_{61}(t)+F_{62}(t)
				,
			\end{equation*}
			where
			\begin{equation}\label{e.w10237}
				F_{61}(t)\lesssim \frac{1}{2}F_1(t) +O(E^{1/2})F_1(t) +M_{6}^{-2}+E^2(t) F_1(t)
				,
			\end{equation}
			and
			\begin{equation}\label{e.w10238}
				F_{62}(t)\lesssim E(t)+C_{0}^{-1}F_2(t)
				.
			\end{equation}
			Thus, combining the above estimates in equations \eqref{e.w10231},\eqref{e.w10233},\eqref{e.w10234},\eqref{e.w10235},\eqref{e.w10237}, and \eqref{e.w10238} yields
			\begin{align*}
				\partial_t E(t) + F_1 (t) + F_2 (t)
				&\leq 
				F_3 (t) +E(t)+M_{6}^{-2}+M_{0}^{-1} F_{2}(t)+\frac{1}{2}F_1(t) +O(E)^{1/2}F_1(t) +M_{6}^{-2}
				\\&\quad
				+E^2(t) F_1(t)+ F_{10} (t)
			\end{align*}	
			and thus 
			\begin{equation*}
				\partial_t E(t) + F_1 (t) + F_2 (t)
				\leq 
				F_3 (t) +E(t)+M_{6}^{-2} +O(E^{1/2})F_1(t) +M_{6}^{-2}+E^2(t) F_1(t)+ F_{10} (t)
			\end{equation*}	
			Finally, by equations \eqref{e.w10231} and \eqref{e.w10236} and continuity argument, we get	for $t_1 \leq t \leq 1$
			\begin{equation*}
				E(t)\lesssim M_{6}^{-2}
		\end{equation*}
		and 
			$$
			\int_{t_1}^1 F_1 (t) dt = \int_{t_1}^1 \int_{\R^3} |\nabla \omega^{\textrm{nlin}} (t,x)|^2 \theta(t,x) dx dt \lesssim M_6^{-2}.
			$$			
			which ends the proof of Lemma.
	\end{proof}
	
	With the estimates of the vorticity, we are able to show the estimates of the velocity in Proposition \ref{p.w08294}.
		\begin{proof}[Proof of Proposition \ref{p.w08294}]
			Using the same Whitney decomposition argument as in \cite{tao} and combing the results in Lemma \ref{Lemma_enstropy}, we have the estimates of the velocity 
			$$
			\sup_{t_1 \leq t \leq 1} \int_{\mathcal{I}_{M_6^7R}} |\nabla u^{\textrm{nlin}} (t,x)|^2  dx dt \lesssim M_6^{-2}
			$$
			and
			$$
			\int_{t_1}^1 \int_{\mathcal{I}_{M_6^6 R}} |\nabla^2 u^{\textrm{nlin}} (t,x)|^2 dx dt \lesssim M_6^{-2},
			$$
			where we recall the definition of $\mathcal{I}_{M_6^{k} R} := \{x: M_6^{-k}R\leq |x| \leq M_6^{k}R\}$ and $k \in \mathbb N$.
			Then, from the Gagliardo-Nirenberg inequality for $\| u^{\nlin}\|_{L^{\infty}_{x} }$ and Hölder inequality, we have
			\begin{align*}
				\|u^{\nlin}\|_{L_t^4L^{\infty}_x  ([t_1, 1] \times \mathcal{I}_{M_6^5R})}  
				\lesssim 
				& \, \left\| \|\nabla u^{\nlin} \|^{1/2}_{L^{2}_{x}(  \mathcal{I}_{M_6^5R})}  \|\nabla^2 u^{\nlin} \|^{1/2}_{L^{2}_{x}(  \mathcal{I}_{M_6^5 R})}\right\|_{L_t^4 ([t_1, 1])} \\
				\lesssim
				& \, \|\nabla u^{\nlin} \|^{1/2}_{L^\infty_t L^{2}_{x}( [t_1, 1] \times \mathcal{I}_{M_6^5R})} \|\nabla^2 u^{\nlin} \|^{1/2}_{L^2_t L^{2}_{x}( [t_1, 1] \times\mathcal{I}_{M_6^5 R})} \lesssim  \, M_6^{-2}.
			\end{align*}
			Combining with the estimates of the linear part $u^{\lin}$ in Lemma \ref{Lemma_linear}, we get that
			\begin{equation*}
				\|u\|_{L_t^4L^{\infty}_x  ([t_1, 1] \times\mathcal{I}_{M_6^5R} )} \lesssim M_6^{-2}.
			\end{equation*}
			Notice that $L^4_tL^\infty_x$ are sub-critical regularity estimate, so by using the same argument as in (iii), as well as the local version of multiplier theorem (see \eqref{e:multiplier2}), we can obtain higher regularity
				\begin{equation*}
					\|u\|_{L_t^8 L^{\infty}_x  ([t_1, 1] \times \mathcal{I}_{M_6^4 R}) } \lesssim M_6^{-2}
				\end{equation*}
Then we fix the time interval but shrink the space interatively by  \eqref{e:multiplier2} and obtain
				\begin{equation*}
						\|u\|_{L_t^\infty L^{\infty}_x  ([t_1, 1] \times \mathcal{I}_{M_6^3 R})} \lesssim M_6^{-2}.
					\end{equation*}
				Repeat the third step in (iii), it follows 
				$$
				\|\nabla u\|_{L_t^4 L^{\infty}_x  ([t_1, 1] \times \mathcal{I}_{M_6^2 R}) } \lesssim M_6^{-2}\comma
				\|\nabla u\|_{L_t^\infty L^{\infty}_x  ([t_1, 1] \times \{M_6 R\leq |x|\leq 2 M_6 R \}) } \lesssim M_6^{-2}.
				$$
				Using vorticity's equation, we get 
				$$
				\|\nabla \omega\|_{L_t^\infty L^{\infty}_x  ([t_1, 1] \times \{R\leq |x|\leq M_6 R \}) } \lesssim M_6^{-2}.
				$$
	\end{proof}
		\section{Proofs of Theorems \ref{t.w08292}, \ref{t.w08291}, and \ref{t.w08293}}\label{proofs}
		In this section, we prove our main theorems. We start with the proof of Theorem \ref{t.w08293} and apply it to prove Theorems \ref{t.w08291} and \ref{t.w08292}.
			\begin{theorem}\label{t.w08293}
				Assume that $t_0, T, u, p, M$ obey the hypotheses of Propositions \ref{prop : bounded}, \ref{p.w08291}-\ref{p.w08294} and that there exists $x_0 \in \R^3$ and $N_0 >0$ such that
				\begin{equation*}
				|P_{N_{0}}u(t_0 , x_0)| \geq M_{1}^{-1}N_0
				.
				\end{equation*}
				Then,
				\begin{equation*}
				TN_{0}^{2}\leq \exp(\exp( \exp (M_{6}^{O(1)})))
				.
					\end{equation*}
			\end{theorem}
			\begin{proof}[Proof of Theorem \ref{t.w08293}]
				The proof is by contradiction and similar to \cite{tao}. The idea is to apply the quantitative version of the Carleman estimates from \cite{ESS} (see \ref{carleman}, \ref{carleman1}, and \ref{carleman2}), which requires Propostion \ref{p.w08291} (epochs of regularity) and Proposition \ref{p.w08294} (annuli of regularity) to provide good quantitative estimates. After summing the disjoint scales, we will obtain a contradiction to \eqref{e:ub-bounds}.
			\end{proof}
		After we prove Theorem \ref{t.w08293}, we are ready to prove Theorem \ref{t.w08292} and Theorem \ref{t.w08291}.
		\begin{proof}[Proof of Theorem \ref{t.w08291}]	By rescaling, it suffices to prove the result when $t=1$, so we have $T \geq 1$. 
			Without loss of generality, 
			we assume that $M \geq C_0$, so Theorem \ref{t.w08293} implies that, for $N \geq N_* := \exp(\exp( \exp (M_{7})))$
				\begin{equation}\label{esti_para}
					\|P_N u\|_{L^\infty_t L^\infty_x ([1/2, 1] \times \R^3)} \leq M_{1}^{-1}N.
				\end{equation}
				On $[1/2, 1] \times \R^3$, we split $u= u^{\textrm {lin}} + u^{\textrm {nlin}}$ where $u^{\textrm {lin}}$ is the linear solution 
				$$
				u^{\textrm {lin}} := e^{t \Delta} u (0) 
				$$
				and $u^{\textrm {nlin}} := u - u^{\textrm {lin}}$ is the nonlinear component, and similarly, we split $\omega= \omega^{\textrm {lin}} + \omega^{\textrm {nlin}}$. From the standard heat kernel bound \eqref{e:bernsteintime} in Lorentz spaces and hypothesis \eqref{e:ub-bounds}, we have
				\begin{align*}\label{estimate : ulin}
					\|\nabla^{j}u^{\lin}\|_{L_{t}^{\infty}L_{x}^{p,q_1} ([1/2, 1] \times \R^3)}
					\lesssim
					M,
				\end{align*}
				where $j\geq 0$ and $3\leq p \leq \infty$, $1/q_{1}\leq 1/p+1/\mathfrak{q}_{0}+2/3$. In particular, we have
				\begin{equation}\label{estimate : ulin_leb}
					\|\nabla^{j}u^{\lin}\|_{L_{t}^{\infty}L_{x}^{p} ([1/2, 1] \times \R^3)}
					\lesssim M \quad \text{and} \quad \|\nabla^{j}\omega^{\lin}\|_{L_{t}^{\infty}L_{x}^{p} ([1/2, 1] \times \R^3)}
					\lesssim M
				\end{equation}
				for all $j\geq 0$ and $3\leq p \leq \infty$.
				As in the proof of Proposition \ref{p.w08291}, for the vorticity, we define the following nonlinear enstrophy-type quantity for $t \in [1/2, 1]$, 
				$$
				F(t) = \frac{1}{2} \int_{\R^3} |\omega^{\textrm{nlin}} (t,x)|^2 d x.
				$$
				By Plancherel's theorem, we have 
				\begin{equation*}
					\|\nabla u^{\textrm{nlin}}\|_{L^2 (\R^3)} \lesssim \| \omega^{\textrm{nlin}}\|_{L^2 (\R^3)} = \sqrt{2} F(t)^{1/2}
				\end{equation*}
				From the vorticity equation \eqref{eq : vorticity} and integrating by part we obtain
				$$
				\partial_t F(t) = - F_1 (t) - F_2 (t)+ F_3 (t) + F_4 (t) + F_5 (t) + F_6 (t)
				$$
				where 
				$$F_1 (t) = \int_{\R^3} |\nabla \omega^{\textrm{nlin}} (t,x)|^2 dx, \quad F_2 (t) = - \int_{\R^3} \omega^{\textrm{nlin}} \cdot (u \cdot \nabla) \omega^{\textrm{lin}} dx, $$
				$$F_3 (t) = \int_{\R^3}  \omega^{\textrm{nlin}} \cdot (\omega^{\textrm{nlin}}  \cdot \nabla) u^{\textrm{nlin}} dx, \quad F_4 (t) = \int_{\R^3} \omega^{\textrm{nlin}} \cdot (\omega^{\textrm{nlin}} \cdot \nabla) u^{\textrm{lin}} dx, $$
				$$F_5 (t) = \int_{\R^3} \omega^{\textrm{nlin}} \cdot (\omega^{\textrm{lin}} \cdot \nabla) u^{\textrm{nlin}} dx, \quad F_6 (t) = \int_{\R^3} \omega^{\textrm{nlin}} \cdot (\omega^{\textrm{lin}} \cdot \nabla) u^{\textrm{lin}} dx.$$ 
				Among these six terms, the third term $F_3(t)$ is more delicate to treat as there are three non-linear terms involving in. For terms $F_2 (t)$ and $F_6(t)$, we use Hölder's inequality to get that, for $t \in [1/2, 1]$, 
				$$
				F_2 (t) \leq \|\omega^{\textrm{nlin}} \|_{L^2 (\R^3)} \|u\|_{L^4 (\R^3)} \| \nabla \omega^{\textrm{lin}}\|_{L^4 (\R^3)} \lesssim M^2 F(t)^{1/2} \leq M^4 + F(t)
				$$
				and 
				$$
				F_6 (t) \leq \|\omega^{\textrm{nlin}} \|_{L^2 (\R^3)} \|\omega^{\textrm{lin}} \|_{L^4 (\R^3)} \| \nabla u^{\textrm{lin}}\|_{L^4 (\R^3)} \lesssim M^2 F(t)^{1/2} \leq M^4 + F(t)
				$$
				where we used estimate \eqref{estimate : ulin_leb} and Cauchy-Schwarz inequality. Similarly, for $F_4 (t)$ and $F_5(t)$, we have for $t \in [1/2, 1]$, 
				$$
				F_4 (t) \leq \|\omega^{\textrm{nlin}} \|_{L^2 (\R^3)} \|\omega^{\textrm{nlin}} \|_{L^2 (\R^3)} \| \nabla u^{\textrm{lin}}\|_{L^\infty (\R^3)} \lesssim M F(t)
				$$
				and
				$$
				F_5 (t) \leq \|\omega^{\textrm{nlin}} \|_{L^2 (\R^3)} \|\omega^{\textrm{lin}} \|_{L^\infty (\R^3)} \| \nabla u^{\textrm{nlin}}\|_{L^2 (\R^3)} \lesssim M F(t).
				$$			
			We now turn to deal with $F_3 (t)$ by using the bound \eqref{esti_para}. By Littlewood-Paley decomposition
			\begin{equation*}
				F_3(t)\lesssim \sum_{N_1, N_2, N_3} \displaystyle\int_{\mathbb{R}^3} P_{N_1}\omega^{\textrm {nlin}}\cdot(P_{N_2} \omega^{\textrm {nlin}}\cdot \nabla ) P_{N_3}u^{\textrm {nlin}}~dx,
			\end{equation*}
			where $N_1, N_2, N_3$ range over powers of two. The integral does not vanish in three cases: $N_1\sim N_2 \gtrsim N_3$, $ N_2 \sim N_3\gtrsim N_1$ and $N_1\sim N_3\gtrsim N_2$.
			Then we control the two highest frequency terms in $L^2_x$ and the lower one in $L^{\infty}_x$, and by the Littlewood-Paley,  we obtain
			\begin{equation*}
				F_3(t) \lesssim \sum_{N_1, N_2,N_3: N_1\sim N_2 \gtrsim N_3} \|P_{N_1} \omega^{\textrm {nlin}}\|_{L_x^2(\mathbb{R}^3)} \|P_{N_2} \omega^{\textrm {nlin}}\|_{L_x^2(\mathbb{R}^3)} \|P_{N_3} \omega^{\textrm {nlin}}\|_{L_x^{\infty}(\mathbb{R}^3)}.
			\end{equation*}
			By \eqref{e:pnomjbounds}, 
			\begin{equation*}
				\|P_{N_3} \omega^{\textrm {nlin}}\|_{L_x^{\infty}(\mathbb{R}^3)} \lesssim O(MN_3^2).
			\end{equation*}
			If $N_3\geq N_*$, by\eqref{esti_para} and \eqref{e:bernstein} we have ($j=1$, $p_1=q_1=p_2=q_2=\infty$)
			\begin{equation*}
				\|P_{N_3} \omega^{\textrm {nlin}}\|_{L_x^{\infty}(\mathbb{R}^3)} =	\|P_{N_3} \nabla \times u^{\textrm {nlin}}\|_{L_x^{\infty}(\mathbb{R}^3)} \lesssim N_3\|P_{N_3} u^{\textrm {nlin}}\|_{L_x^{\infty}(\mathbb{R}^3)}\lesssim O(M_1^{-1}N_3^2).
			\end{equation*} 
			We thus obtain
			\begin{align*}
				\sum_{N_3\lesssim N_2} \|P_{N_3} \omega^{\textrm {nlin}}\|_{L_x^{\infty}(\mathbb{R}^3)} & \lesssim \sum_{N_*<N_3\leq N_2}M_1^{-1}N_3^2+\sum_{N_3<N_*}MN_3^2 \\
				& \lesssim M_1^{-1}N_2^2+MN_{*}^2
			\end{align*}
			and by Cauchy-Schwarz, 
			\begin{equation*}
				F_3(t)\lesssim \sum_{N_1} \|P_{N_1} \omega^{\textrm {nlin}}\|_{L_x^2(\mathbb{R}^3)}^2 (M_1^{-1}N_1^2+MN_{*}^2).
			\end{equation*}
			On the other hand, by Plancherel's theorem, 
			\begin{equation*}
				F_1(t)=\|\nabla \omega^{\textrm{nlin}} (t,x)\|_{L_x^2(\mathbb{R}^3)}^2\sim \sum_{N_1} \|P_{N_1} \omega^{\textrm {nlin}}\|_{L_x^2(\mathbb{R}^3)} N_1^2.
			\end{equation*}
			Recall that
$			F(t)=\frac{1}{2}\int_{\R^3} |\omega^{\nlin}|^{2}\,dx $, then by Plancherel's theorem, we get
			\begin{equation*}
				F(t)\sim \sum_{N_1}\|P_{N_1}\omega^{\nlin}\|^{2}_{L_{x}^{2}(\R^3)}
				.
			\end{equation*}
			Therefore,
			\begin{equation*}
				F_{3}(t) \lesssim M_{1}^{-1}F_{1}(t) +MN_{*}^{2}F(t)
				.
			\end{equation*}
			Combing the above estimates yields
			\begin{equation}\label{e.w09121}
				\partial_{t}F(t) +F_{1}(t)
				\lesssim
				M^4 +F(t) +M_{1}^{-1}F_{1} +MN_{*}^{2}F(t) +MF(t)
				\lesssim
				MN_{*}^{2}F(t) +M^4
				.
			\end{equation}
			Integrating from $t_{1}$ to $t_{2}$ with $1/2 <t_1 <t_2 <1$ and $|t_2 -t_1|\leq M^{-1}N_{*}^{-2}$ and applying Gronwall gives
			\begin{equation}\label{e.w09122}
				F(t_2)\lesssim F(t_1) +M^4
				.
			\end{equation}
			Next, by \eqref{e.w07019}, we see
			\begin{equation*}
				\int_{1/2}^{1} F(t)\,dt
				=
				\frac{1}{2}\int_{1/2}^{1}\int_{\R^3} |\omega^{\nlin}|^{2}\,dxdt
				\lesssim
				M^4
				.
			\end{equation*}
			Therefore, by the pigeonhole principle, we see on any time interval in $[1/2,1]$ of length $M^{-1}N_{*}^{-2}$, there exists at least one time $t$ such that $
				F(t)
				\lesssim
				M^4 $.
			Thus, for all time $t \in [3/4,1]$, we obtain
			\begin{equation}\label{e.w09123}
				F(t)
				\lesssim
				M^{5}N_{*}^{2}
				\lesssim
				N_{*}^{O(1)}
				.
			\end{equation}
			Then, by the fundamental theorem of calculus together with equations \eqref{e.w09121}, \eqref{e.w09122}, and \eqref{e.w09123}, we get
			\begin{equation*}
				\int_{3/4}^{1} F_{1}(t)\,dt\lesssim N_{*}^{O(1)}
				.
			\end{equation*}
			Now we conclude the proof by appealing to Proposition \ref{p.w08291}.
		\end{proof}
		Next, we prove Theorem \ref{t.w08292}.
		\begin{proof}[Proof of Theorem \ref{t.w08292}]
			We argue by contradiction. First, we rescale $T_* =1$. Suppose
			\begin{equation*}
				\limsup_{t\rightarrow 1+}\frac{\|u\|_{L^{3, \mathfrak{q}_{0} }(\R^3)}}{(\log\log\log\frac{1}{1-t})^c}
				<
				\infty
				,
			\end{equation*}
			where $c>0$ is a small constant. Then, for some constant $C>0$ we have for $0< t \leq 1$
			\begin{equation*}
				\|u\|_{L^{3, \mathfrak{q}_{0} }(\R^3)}
				\leq
				C\left(\log\log\log\left(100+\frac{1}{1-t}\right)\right)^c
				.
			\end{equation*}
			Then, by Theorem \ref{t.w08291}, we get for $j=0,1$ and for $0<1/2 \leq t \leq 1$
			\begin{equation*}
				\begin{split}
					\|\nabla^{j}u\|_{L^{\infty }(\R^3)}
					&\lesssim
					\exp{\exp{\exp{(M^{O(1)})}}}t^{-(j+1)/2}
					\\&\lesssim
					\exp{\exp{\exp{\left(\log\log\log\left(100+\frac{1}{1-t}\right)\right)^{O(1)}}}}t^{-(j+1)/2}
					\leq
					C(1-t)^{-1/10}.
				\end{split}
			\end{equation*}
			Similarly, we have
			\begin{equation*}
				\|\nabla^{j}\omega\|_{L^{\infty }(\R^3)}
				\leq
				C(1-t)^{-1/10}
				.
			\end{equation*}
			This implies that $u\in L_{t}^{2}L_{x}^{\infty}$ contradicting the blow-up criterion by Prodi-Ladyzhenskaya-Serrin.
		\end{proof}


		\appendix


		\section{Carleman estimates}
		\begin{theorem}\label{carleman}
			(General Carleman inequality) Let $[t_1, t_2]$ be a time interval, and let $u \in C_{c}^{\infty}([t_1, t_2]\times \R^d \rightarrow \R^m)$ be a vector-valued test function solving the backwards heat equation
			\begin{equation*}	
				Lu=f
				.
			\end{equation*}
			with $L$ the backwards heat operator
			\begin{equation*}	
				L:=\partial_t + \Delta
				.
			\end{equation*}
			and let $g:[t_1,t_2]\times \R^d \rightarrow \R$ denote the function
			\begin{equation*}	
				F:=\partial_t g - \Delta g -|\nabla g|^2
				.
			\end{equation*}  
			Then we have the inequality
			\begin{equation*}	
				\partial_t \int_{\R^d} \left(|\nabla u|^2 + \frac{1}{2}F|u|^2 \right)e^g \,dx
				\geq
				\int_{\R^d} \left(\frac{1}{2}(LF)| u|^2 + 2D^{2}g(\nabla u, \nabla u)-\frac{1}{2}|Lu|^2 \right)e^g \,dx
			\end{equation*}  
			for all $t \in I$, where $D^2 g$ is the bilinear form expressed in coordinates as 
			\begin{equation*}	
				D^2 g (v,w):= (\partial_i \partial_j g)v_i \cdot w_j
			\end{equation*} 
			with the usual summation conventions. In particular, from the fundamental theorem of calculus one has
			\begin{equation*}	
				\int_{t_1}^{t_2} \int_{\R^d} \left(\frac{1}{2}(LF)| u|^2 + 2D^{2}g(\nabla u, \nabla u) \right)e^g \,dx
				\geq
				\frac{1}{2}\int_{t_1}^{t_2}\int_{\R^d} | Lu|^2  e^g \,dx
				+
				\int_{\R^d} \left(|\nabla u|^2 + \frac{1}{2}F|u|^2 \right)e^g \,dx|_{t=t_1}^{t=t_2}
				.
			\end{equation*} 
		\end{theorem}
		\begin{theorem}\label{carleman1}
			(First Carleman inequality) Let $T>0$, $0<r_- <r_+$, and let $\mathcal A$ denote the cylindrical annulus
			\begin{equation*}	
				\mathcal A:= \{(t,x)\in \R\times \R^3: t \in [0,T]; r_- \leq |x| \leq r_+  \}
				.
			\end{equation*}
			Let $u:\mathcal A \rightarrow \R^3$ be a smooth function obeying the differential inequality
			\begin{equation*}	
				|Lu| \leq C_{0}^{-1} T^{-1}|u| +C_{0}^{-1/2}T^{-1/2}|\nabla u|
			\end{equation*}
			on $\mathcal A$. Assume the inequality
			\begin{equation*}	
				r_{-}^{2} \geq 4C_{0}T.
			\end{equation*}
			Then one has
			\begin{equation*}	
				\int_{0}^{T/4} \int_{10r_{-}\leq |x| \leq r_{+}/2} (T^{-1}|u|^{2} + |\nabla u|^{2})\,dxdt
				\lesssim
				C_{0}^{2} e^{-\frac{r_{-}r_{+}}{4C_{0}T}}(X+e^{2r_{+}^{2}/C_{0}T}Y)
				,
			\end{equation*}
			where
			\begin{equation*}	
				X:= \int \int_{\mathcal A} e^{2|x|^{2}/C_{0}T}(T^{-1}|u|^{2} + |\nabla u|^{2})\,dxdt
			\end{equation*}
			and
			\begin{equation*}	
				Y:=\int_{r_{-}\leq |x| \leq r_{+}}|u(0,x)|^{2}\,dx.
			\end{equation*}
		\end{theorem}
		\begin{theorem}\label{carleman2}
			(Second Carleman inequality) 
			Let $T,r>0$ and let $\mathcal C$ denote the cylindrical region. Assume the inequality
			\begin{equation*}	
				\rho^2 \geq 4000T
				.
			\end{equation*}
			Then for any
			\begin{equation*}	
				0<t_1 \leq t_0 <\frac{T}{1000}
			\end{equation*}
			one has
			\begin{equation*}	
				\int_{t_0}^{2t_0} \int_{|x| \leq \rho/2} (T^{-1}|u|^{2} + |\nabla u|^{2})\,dxdt
				\lesssim
				Xe^{-\frac{\rho^2}{500T}}
				+
				t_{0}^{3/2}(et_{0}/t_1)^{O(\rho^2/t_0)}Y
				,
			\end{equation*}
			where
			\begin{equation*}	
				X:= \int_{0}^{T} \int_{|x|\leq \rho} (T^{-1}|u|^{2} + |\nabla u|^{2})\,dxdt
			\end{equation*}
			and
			\begin{equation*}	
				Y:=\int_{|x| \leq \rho}|u(0,x)|^{2}t_{1}^{-3/2}e^{-|x|^{2}/4t_1}\,dx.
			\end{equation*}
		\end{theorem}

	\section*{Acknowledgement}
	This work was partially supported by the NSF Grant No. DMS-1928930 while the three authors participated in the Mathematical Fluid Dynamics program hosted by the Mathematical Sciences Research Institute in Berkeley, California, during the Spring 2021 semester. A part of this work was written when JH was in Evry, supported by the Sophie Germain Post-doc program of the Fondation Mathematique Jacques Hadamard, and she thanks Pierre Gilles Lemarié-Rieusset and Christophe Prange for their helpful discussions.
	WW was partially supported by an AMS-Simons travel grant and he is grateful to Tai-Peng Tsai for helpful discussions.
	
\section*{Declarations}
Competing interests : On behalf of all authors, the corresponding author states that there is no conflict of
interest.

Availability of data and material : Not applicable.

Code availability : Not applicable.

	\bibliographystyle{siam}
	\bibliography{Feng-He-Wang}
	
\end{document}